\documentclass[11pt]{amsart}

\usepackage{amsmath}
\usepackage{amsthm}
\usepackage{amsfonts}
\usepackage[english]{babel}
\usepackage[margin=1.1in]{geometry}
\usepackage[colorlinks=true, linkcolor=blue]{hyperref}

\newtheorem{prop}{Proposition}
\newtheorem{lemma}[prop]{Lemma}

\newtheorem{thm}[prop]{Theorem}
\newtheorem{cor}[prop]{Corollary}

\numberwithin{prop}{section}
\numberwithin{equation}{section}

\theoremstyle{definition}
\newtheorem{defn}[prop]{Definition}

\newtheorem{rmk}[prop]{Remark}

\def \k {K\"ahler }
\newcommand{\oo}[1]{\overline{#1}}

\newcommand{\hsp}{\hspace{0.5cm}}

\newcommand{\vp}{\varphi}

\newcommand{\del}{\partial}
\newcommand{\bdel}{\bar{\partial}}

\newcommand{\ga}{\alpha}
\newcommand{\gb}{\beta}

\newcommand{\gw}{\omega}

\newcommand{\ten}{\otimes}
\newcommand{\w}{\wedge}
\newcommand{\bC}{\mathbb{C}}
\newcommand{\bCn}{\mathbb{C}^n}
\newcommand{\dd}{\partial}
\newcommand{\dbar}{\oo\partial}

\newcommand{\modt}{\equiv_2}
\newcommand{\nmodt}{\not\equiv_2}
\newcommand{\eps}{\varepsilon}
\newcommand{\delb}{\overline{\del}}

\DeclareMathOperator{\supp}{supp}

\DeclareMathOperator{\Ric}{Ric}

\DeclareMathOperator{\Rm}{Rm}

\DeclareMathOperator{\Real}{Re}
\DeclareMathOperator{\dist}{dist}

\begin{document}

\title[Asymptotic expansion of Bergman kernel]{Asymptotic expansion of the Bergman kernel via perturbation of the Bargmann-Fock model}

\begin{abstract}
We give an alternate proof of the existence of the asymptotic expansion of the Bergman kernel associated to the $k$th tensor powers of a positive line bundle $L$ in a $\frac{1}{\sqrt{k}}$-neighborhood of the diagonal using elementary methods.  We use the observation that after rescaling the K\"ahler potential $k\vp$ in a $\frac{1}{\sqrt{k}}$-neighborhood of a given point, the potential becomes an asymptotic perturbation of the Bargmann-Fock metric.  We then prove that the Bergman kernel is also an asymptotic perturbation of the Bargmann-Fock Bergman kernel.
\end{abstract}

\date{\today}

\author [Hezari]{Hamid Hezari}

\email{\href{mailto:hezari@uci.edu}{hezari@uci.edu}}

\author [Kelleher]{Casey Kelleher}

\email{\href{mailto:clkelleh@uci.edu}{clkelleh@uci.edu}}

\author [Seto]{Shoo Seto}

\email{\href{mailto:shoos@uci.edu}{shoos@uci.edu}}

\author [Xu]{Hang Xu}

\email{\href{mailto:hangx@uci.edu}{hangx@uci.edu}}

\address{Rowland Hall\\
         University of California\\
         Irvine, CA 92697}

\maketitle
\section{Introduction}

Let $(L,h) \to (M^{n},\omega)$ be a positive Hermitian holomorphic line bundle over a compact complex manifold. The hermitian metric $h$ induces a \k form $\omega$ on $M$
\begin{equation*}
\gw := -\tfrac{\sqrt{-1}}{2\pi} \del \bdel \log(h).
\end{equation*}
Let $H^0(M,L)$ denote the space of holomorphic global sections of $L$, which is a closed subspace of $\mathcal{L}^2(M,L)$, the space of all square integrable sections of $L$ over $M$. The $\mathcal{L}^2$-inner product of $f,g \in H^0(M,L)$ is defined by
\begin{equation*}
\langle  f,g \rangle_{\mathcal{L}^2} :=  \int_M (f,g)_h \tfrac{\omega^n}{n!}.
\end{equation*}
The \emph{Bergman projection} is the orthogonal projection $\mathcal{P}_{H^0}: \mathcal{L}^{2}(M,L) \to H^0(M,L)$.
The \emph{Bergman kernel} $K$, a section of $L \ten \bar{L}$, is the operator kernel of $\mathcal{P}_{H^0}$ with respect to the above inner product.

The Bergman kernel has the following \emph{(global) reproducing property}: given a holomorphic section $f \in H^0(M,L)$ we have
\begin{equation*}
f(x) = \left\langle f(y),K(y,x) \right\rangle_{\mathcal{L}^2}.
\end{equation*}
Given $x_0 \in M$, a sufficiently small neighborhood $U_{x_0}$ admits a local trivialization of $L$ with frame $e_L$ on $U_{x_0}$. We define the local \k potential $\varphi$ by
\begin{equation*}
h(e_L,e_L) = e^{-\varphi}.
\end{equation*}
\emph{B\"ochner coordinates} are special coordinates in which the local \k potential $\vp$ admits the form
\begin{equation*}
\vp(z) = |z|^2 + R(z), \ R(z) =  O(|z|^4).
\end{equation*}

There always exists a neighborhood of $p$ which admits B\"ochner coordinates. Note the above definitions are naturally extended to $ (L^{\otimes k},h^k)$ and we denote the corresponding Bergman Kernel by $K_k$. In this setting, the frame of $L^{\ten k}$ on $U_{x_0}$ is given by $e_L^{\ten k}$, the $k$-tensor product of the frame $e_L$. We shall consider only $K_k$ for the remainder of the paper and will henceforth drop the $k$ subscript.

The purpose of this paper is to provide an alternate proof to the following theorem.

\begin{thm}[\cite{zelditch}, \cite{catlin}, \cite{shiffzeld}]\label{mainresult}
The scaled Bergman kernel admits the following asymptotic expansion in the B\"ochner coordinates and in the frame $e^{\otimes k}_L(x) \otimes e_L^{\otimes k}(\oo y)$,
\begin{equation*}
K \left( \tfrac{u}{\sqrt{k}}, \tfrac{v}{\sqrt{k}}\right) \sim k^n e^{u\cdot \oo v}
\left( 1 + \sum_{j=2}^{\infty}\tfrac{c_j(u,\bar{v})}{\sqrt{k^j}} \right), \ \ \ \ |u|, |v| \leq 1.
\end{equation*}
More precisely, for any $N$, the following inequality holds:
\begin{equation*}
\left\| K \left( \tfrac{u}{\sqrt{k}}, \tfrac{v}{\sqrt{k}}\right) - k^n e^{u\cdot \oo v}
\left( 1 + \sum_{j=2}^{N}\tfrac{c_j(u,\bar{v})}{\sqrt{k^j}} \right) \right\|_{C^m} \leq C_{N,m}k^{n-\frac{N+1-m}{2}},
\end{equation*}
where the $C^m$ norm corresponds to the $x_0,u,v$ variables, and with respect to a fixed B\"ochner coordinate in a smooth family of B\"ochner coordinates centered at $x_0$.

\noindent Here each $c_j(u,\bar{v})$ is a polynomial of the form $\sum_{p,q}c_j^{p,q}(x_0)u^p\bar{v}^q$ satisfying
\begin{equation*}
\begin{cases}
c_j^{p,q}(x_0) = 0 &\text{ for } p+q > 2j,\\
c_j^{p,q}(x_0) = 0 &\text{ for } p+q \neq j \mod 2.
\end{cases}
\end{equation*}
In particular, by setting $u=v=0$, this verifies the on-diagonal expansions of Zelditch \cite{zelditch} and Catlin \cite{catlin}.
\end{thm}

\noindent The primary objective of this paper is to provide a direct proof of Theorem \ref{mainresult}. The uniqueness of our approach is that it uses straightforward computational techniques and requires little advanced construction. The only advanced tool we use is H\"ormander $\mathcal{L}^2$-estimates.

Methods to compute and analyze the coefficients of the Bergman Kernel have been worked out over the last twenty five years. Initially Tian proved leading asymptotics on the diagonal using the method of peak sections \cite{tian}. A complete expansion was given independently by Zelditch \cite{zelditch} and Catlin \cite{catlin} by using Boutet de Monvel and Sj\"{o}strand parametrix \cite{boutetdemonvel}. The near-diagonal expansion of $K(x,y)$ (i.e. $\dist(x, y) < 1/\sqrt{k})$ is deduced from the diagonal expansion by Taylor expansion (see \cite{shiffzeld} for the more general case of symplectic manifolds). The near-diagonal asymptotic expansions are of the form (with $b_{l}(x,\oo y)$ certain hermitian functions; $b_l(x,\oo y) = \oo{b_l(y,\oo x)}$),
\begin{equation}\label{neardiagonalversion2}
K(x,y) = k^n e^{k \psi(x,\oo y)} \left( 1+ \sum_{l = 1}^{\infty} \frac{b_{l}(x,\oo y)}{k^{l}} \right), \ \dist(x,y) < \frac{1}{\sqrt{k}},
\end{equation}
where in the real analytic case, $\psi(x,\oo y)$ is the polarization of $\vp$, i.e. $\psi(x, \oo y)$ is holomorphic in $x$ and $\oo y$, and $\psi(x,\oo y)|_{y =x} = \varphi(x)$. It is clear that this expansion implies the one in Theorem \ref{mainresult} by Taylor expansion. Lu demonstrated that the functions $b_l(x,\oo x)$ are polynomials of covariant derivatives of the curvature of the underlying manifold $M$ and computed the first four terms \cite{lu}. The off-diagonal terms $c_j(u,\oo v)$ were studied by Lu and Shiffman \cite{lushiff} using Taylor expansions of $b_l(x,\oo x)$. Berman, Berndtsson, and Sj\"ostrand \cite{berman} gave an alternate approach to prove \eqref{neardiagonalversion2}, which we will discuss further in the following section.

\subsection{Sketch of proof.}

Our proof of Theorem \ref{mainresult} is subdivided into two components: construction and analysis of a local reproducing kernel using perturbation methods on the reproducing kernel of the Bargmann-Fock space (\S \ref{refkerbCn}, \S \ref{localestimates}), and demonstration of this construction in fact being (asymptotically) the global reproducing kernel (\S \ref{localtoglobal}). As in \cite{berman}, initially the asymptotic formula is formally computed and then remainder estimates are addressed. However, our local method of computation is distinct; while \cite{berman} gives an approach using microlocal analysis techniques inspired by the calculus of Fourier integral operators with complex phase developed by Sj\"ostrand, we use standard integration and combinatorial identities on $\mathbb{C}^n$ to calculate the coefficients that would satisfy the reproducing property. The proof relating the local reproducing kernel to the global Bergman Kernel is essentially the same as \cite{berman} and is included for completeness. We also explicitly compute the first few coefficients of the expansion in \S \ref{compofcoeff}.
We define the \emph{local reproducing kernel modulo $k^{-(N+1)/2}$} on $U_{x_0}$ to be a function $Q_N(x,y)$ on $U_{x_0} \times U_{x_0}$ which is holomorphic in $x$, antiholomorphic in $y$, and which satisfies the following \emph{local reproducing property} (modulo $k^{-(N+1)/2}$)
\begin{equation*}
f(x) = \left\langle \chi_k(y)f(y), Q_N(y,x) \right\rangle_{\mathcal{L}^2(U_x,k\varphi)} + O \left(k^{n-\tfrac{N+1}{2}}\right) \|f\|_{\mathcal{L}^2(U_x,k\varphi)}, \ f \in H^0(U_{x_0}),
\end{equation*}
where $\chi_k$ is a cutoff function supported in a scale ball of radius $k^{-1/4-\epsilon}$ for some $\eps > 0$ (see \eqref{cutofffunc} for the precise definition). The choice of such an unusual cutoff function with shrinking support plays an important role in our argument.

To construct the local reproducing kernel modulo $N$ at a point $x_0 \in M$, we begin by choosing B\"ochner coordinates and a local trivialization of the bundle. Our cutoff function $\chi_k$ is chosen to have shrinking support on $B(k^{-1/4 -\epsilon})$, so that the inner product is localized near the diagonal and also to ensure that the local rescaled K\"ahler potential admits an asymptotic expansion of the form
\begin{equation*}
 \vp \left( \tfrac{v}{\sqrt{k}} \right) := \left| \tfrac{v}{\sqrt{k}} \right|^2 \left( 1 + \sum_{j=2}^{\infty} \tfrac{a_j(v, \bar{v})}{\sqrt{k}^{j}} \right), \ k \to \infty.
\end{equation*}
This expression is an asymptotic perturbation of $|\frac{v}{\sqrt{k}}|^2$, hence we propose that the local Bergman kernel admits an asymptotic expansion of the form
\begin{equation}\label{eq:BKexpan}
K^{loc} \left( \tfrac{u}{\sqrt{k}}, \tfrac{v}{\sqrt{k}} \right) = k^n e^{u \bar{v}} \left( 1 + \sum_{j=2}^{\infty}\tfrac{c_j(u,\bar{v})}{\sqrt{k}^j} \right),
\end{equation}
where the $c_j$'s depends on $a_j$ and satisfy $c_j(u,\bar{v}) = \overline{c_j(v,\bar{u})}$. The reason we propose such an expansion is that if $\phi(x)=|x|^2$, then $a_j=0$ for all $j \geq 2$ and hence $c_j = 0$ for all $j \geq 2$ yielding $K^{loc} ( \tfrac{u}{\sqrt{k}}, \tfrac{v}{\sqrt{k}})  = k^n e^{u \cdot \oo v}$, which is precisely the rescale of the kernel of Bargmann-Fock metric $k|x|^2$ (cf. \S \ref{repkerC}).

In \S \ref{refkerbCn} we show the existence of the coefficients which will (formally) satisfy the reproducing property on $\mathbb{C}^n$ under the perturbed metric. In \S \ref{localestimates} we show that the local kernel proposed satisfies the local reproducing property through a series of remainder estimates.
In \S \ref{localtoglobal} we show the global Bergman kernel admits an asymptotic expansion by comparing it to our local construction. The crux of the proof is the application of standard H\"ormander $\delb$-estimates as it was done before in \cite{berman}.

With the proof of Theorem \ref{mainresult} complete, in \S \ref{compofcoeff} we then use the algorithm generated within the argument to compute the coefficient  $c_2$ (the coefficients $c_0$ and $c_1$ are computed in Lemmas \ref{csub0is0} and \ref{csub1is0} respectively), yielding the following result.

\begin{prop}[Coefficients of the expansion] One has
\begin{equation*}
\begin{cases}
c_0 &=1, \\
c_1 &= 0, \\
c_2 &= \frac{\rho}{2} - \frac{1}{4}\sum_{i,j,k,l}\Rm_{i\oo j k \oo l}(0)u^i u^k\oo v^j \oo v^l,
\end{cases}
\end{equation*}
where $\Rm$ denotes the Riemannian curvature tensor and $\rho$ denotes the scalar curvature (cf. \eqref{curvatureconv}). In particular, for $u=v=0$ one obtains $c_2 = \frac{\rho}{2}$.
\end{prop}

\begin{rmk}
Our computation of the coefficients is independent of any previous results on the coefficients of the Bergman kernel expansion on the diagonal.
In fact, the iterative computation implemented to compute these quantities captures the essential strategy one may utilize to compute (if they feel so inclined) any desired $c_j$.
\end{rmk}

\subsection{Twisted bundle case}
Let $E \to M$ be a Hermitian holomorphic vector bundle. Consider the twist $E \ten L^{\ten k}$.  In this setting, the Bergman kernel, $K_{k,E}(z,w)$, can be defined as in the introduction and one can study the asymptotic expansion of this kernel as $k \to \infty$.  In fact our local construction follows with the only difference being the volume form and the local to global construction follows similarly by H\"ormander $\mathcal{L}^2$-estimates to extend our result to the twisted bundle case.
\subsection{Notations and conventions}
We now set the conventions which will be used throughout the paper. Let $\mathbb{Z}_+$ denote the collection of all nonnegative integers. Let $\ell \in \mathbb{N}$ and let $\ga, \gb \in \mathbb{Z}_{+}^{\ell}$ such that $\ga :=  (\ga_1, ... , \ga_{\ell})$ and $\gb :=  (\gb_1, ... , \gb_{\ell})$. Define
\begin{equation*}
\ga! := \prod_{i=1}^{\ell}{\ga_i!}, \qquad |\ga| := \sum_{i = 1}^{\ell} \ga_i.
\end{equation*}
A \emph{multiindex binomial coefficient} is defined by the following
\begin{equation*}
\binom{\ga}{\gb} := \prod_{i=1}^{\ell} \binom{\ga_i}{\gb_i}.
\end{equation*}
Note we use the convention that if $q > p$ then $\binom{p}{q} := 0$. Lastly, \emph{multiindex inequalities} will be defined as follows
\begin{equation*}
\ga \leq \gb \Longleftrightarrow \text{$\ga_i \leq \gb_i$ for all $i$},
\end{equation*}
and furthermore
\begin{equation*}
\ga < \gb  \Longleftrightarrow \text{$\ga \leq \gb$ and $\ga \neq \gb $}.
\end{equation*}
For any summations if the ranges are not specified one will assume summation indices range over multiindex values $\mathbb{Z}_+^{\ell}$.

Given $x \in M$ and $r >0$ we set
\begin{equation*}
B_x(r) :=\left\{ y \in M :  \dist(y,x) < r  \right\}.
\end{equation*}
In $ \bCn$ we set $B(r):= B_0(r)$ and $B := B(1)$.

Given local holomorphic coordinates
$\left\{ z_i \right\}_{i = 1}^n$, the volume form is given by
 \begin{equation*}
dV := \left( \frac{\sqrt{-1}}{2\pi}\right)^n dz^1 \wedge d\bar{z}^1 \wedge \cdots \wedge dz^n \wedge d\bar{z}^n.
\end{equation*}
The components of the curvature tensor are given by
\begin{align}\label{curvatureconv}
\begin{split}
&\Rm_{i\oo jk\oo l} := -\frac{\dd^2 g_{i\oo j}}{\dd z^k \dd\oo z^j} + \sum_{s,t}g^{s\oo t}\frac{\dd g_{i\oo t}}{\dd z^k}\frac{\dd g_{s\oo j}}{\dd\oo z^l}, \qquad \text{(Riemannian curvature)} \\
&\Ric_{i\oo j} = \sum_{k,l}g^{k \oo l}R_{i\oo j k\oo l}, \qquad \text{(Ricci curvature)} \\
&\rho = \sum_{i,j}g^{i\oo j}\Ric_{i\oo j}. \qquad \text{(scalar curvature)}
\end{split}
\end{align}
Furthermore we set
\begin{equation*}
\Omega(z) := \det \left( \frac{\del^2}{\del z_i \del \bar{z_j}} \vp (z) \right).
\end{equation*}
Hence in local coordinates
\begin{equation*}
\frac{\omega^n}{n!} = \Omega \, dV.
\end{equation*}

We define the local weighted $\mathcal L ^2$ spaces by
\begin{equation*}
\mathcal{L}^2(U_x,k\vp) := \left\{ f \ : \int_{U_x}{\left| f(z) \right|^2e^{-k\varphi(z)}\frac{\omega^n}{n!}} < \infty \right\},
\end{equation*} and we define its closed subspace of $\mathcal L^2$-holomorphic functions by $H^0(U_x, k\vp)$. 
The Bargmann-Fock space $\mathcal F$ is precisely the space $H^0(\mathbb C^n, |x|^2)$. We refer the reader to the appendix for some facts about $\mathcal F$ (\S \ref{s:appendix}). 

Finally, we use the following $\frac{1}{\sqrt{k}}$-rescaling convention on the potential:
\begin{equation}\label{scaledpotential}
\varphi_k(v) := \varphi\left(\frac{v}{\sqrt{k}}\right).
\end{equation}

\section{Local construction}\label{refkerbCn}
\subsection{Existence of coefficients}

Let $a_m^{r,s}$ be the coefficients in the formal power series expansion of the product
\begin{equation}\label{eq:acoeff}
e^{-kR\left(\frac{v}{\sqrt{k}} \right)} \Omega \left(\tfrac{v}{\sqrt{k}} \right) = \sum_{m=0}^\infty\sum^{2m}_{r+s=0}\frac{a_m^{r,s}v^r\oo v^s}{\sqrt{k^m}}.
\end{equation}

\begin{prop}[Existence of coefficients]\label{coeffexist}
There exist unique coefficients $c_j^{p,q} \in \mathbb{C}$ depending only on the \k potential $\varphi$ such that for any polynomial $F$ and any $N\geq 0$,
\begin{equation}\label{ccoeff}
F \left(\frac{u}{\sqrt{k}} \right) = \int_{\bCn}F \left(\frac{v}{\sqrt{k}}\right)e^{u\cdot\bar{v}-|v|^2} \left(\sum_{t=0}^N\sum_{m+j=t}\sum_{p,q}\sum_{r,s} \frac{c_j^{p,q}a_m^{r,s}}{\sqrt{k^t}}u^p\bar{v}^q v^r\oo v^s  \right)dV.
\end{equation}
Furthermore, the coefficients $c_j^{p,q}$ have the following finiteness and ``parity property'' (c.f. Definition \ref{parprop})
\begin{enumerate}
\item $c_m^{p,q} = 0$ when $|p+q| >2m$,
\item $c_m^{p,q} = 0$ when $|p|+|q| \nmodt m$.
\end{enumerate}
\end{prop}

\begin{rmk}
When $N = 0$ the equation reduces to the reproducing property of the Bargmann-Fock kernel. Also note that the expression in parenthesis is precisely the truncation up to $k^{-N/2}$ of the product
\begin{equation*}
\left(\sum_{j=0}^\infty \frac{c_j(u, \oo v)}{\sqrt{k^j}}\right ) e^{-kR\left(\frac{v}{\sqrt{k}} \right)} \Omega \left(\tfrac{v}{\sqrt{k}} \right). 
\end{equation*}
\end{rmk}

The above proposition is one of the key results of the paper, though we defer the proof until \S \ref{proofofprop} as it is purely algebraic.  But at the heart of the proof stands Lemma \ref{MainLemma} in the appendix, which states that
\begin{equation*}
\int_{\bCn}{\bar{v}^{p}v^q e^{u \cdot \bar{v} - |v|^2} dV} =
\begin{cases}
0 & \text{ if } p > q, \\
\frac{q!}{(q-p)!} u^{q-p} & \text{ if } p \leq q.
\end{cases}
\end{equation*}
In the following section we prove Proposition \ref{locrepprop} using our key Proposition \ref{coeffexist}.

\section{Remainder estimates}\label{localestimates}
\subsection{On the choice of the shrinking radius}
We begin by choosing a point $p \in M$ and choosing some local neighborhood $U_x$ centered at $p$, which admits a local trivialization. We rescale coordinates via the identifications
\begin{equation*}
\begin{cases}
x &= p + \frac{u}{\sqrt{k}},\\
y &= p + \frac{v}{\sqrt{k}}.
\end{cases}
\end{equation*}
We additionally use B\"ochner coordinates (cf. Proposition \ref{expkahlerpot}). Thus the potential scales as follows:
\begin{equation*}
k \vp_k (v) := k \varphi( \tfrac{v}{\sqrt{k}})=\left| v \right|^2 + k R\left(\tfrac{v}{\sqrt{k}} \right).
\end{equation*}
 Under B\"ochner coordinates we have $R(z) = O(|z|^4)$. We fix $0< \eps < \frac{1}{4}$ for the remainder of the paper, and we further stipulate that $v$ will be chosen from within $B(k^{\frac{1}{4}-\eps})$ to ensure that $kR\left( \frac{v}{\sqrt{k}} \right) = O(k^{-\eps})$. Consequently, for large $k$, one may compute in the setting of a \emph{perturbed Bargmann-Fock space}. In particular, this implies that for $k$ sufficiently large
\begin{equation}\label{eqnorms}
\frac{1}{2} ||f||_{\mathcal L^2(B(k^{-1/4-\eps}),\, k|z|^2)} \leq ||f||_{\mathcal L^2(B(k^{-1/4-\eps}),\, k\vp(z) )}  \leq 2 ||f||_{\mathcal L^2(B(k^{-1/4-\eps}),\, k|z|^2)}.
\end{equation}
To develop this idea effectively, we choose some cutoff function $\chi \in C_c^{\infty}(\bCn)$ satisfying
\begin{equation}\label{cutofffunc}
\chi(x) =
\begin{cases}
1 & \mbox{ if } |x|\leq \frac{1}{2}, \\
0 & \mbox{ if } |x|\geq 1,
\end{cases}
\end{equation}
and set $\chi_k(x) := \chi(k^{\frac{1}{4}+\eps}x)$.
We then incorporate $\chi_k$ into the integral in the local reproducing property of Proposition \ref{locrepprop}.  Note that
\begin{equation*}
\supp(d\chi_k(\tfrac{z}{\sqrt{k}})) \subset \{ z \ | \ \tfrac{1}{2}k^{\frac{1}{4}-\eps} \leq |z| \leq k^{\frac{1}{4}-\eps} \}.
\end{equation*}
So for $|u| \leq 1$ and $v \in \supp(d\chi_k(\tfrac{v}{\sqrt{k}}))$, their distance has a lower bound $|u-v| \geq \frac{1}{4}k^{\frac{1}{4}-\eps}$, which is crucial in obtaining an estimate for the exponential decay outside the near-diagonal neighborhood. We then analyze the orders of the remainders of the truncations of the locally defined pieces to show that their contribution to the local reproducing property is of negligible order.

\subsection{Local reproducing kernel}

In this section we show that local reproducing kernel with the coefficients chosen in \S \ref{refkerbCn} actually satisfies the local reproducing property up to a small error, and give measure of such error. Before doing so, we introduce some notation. Given a domain $U \subset \mathbb{C}^n$ and $f \in C^\infty(U)$ we set
\begin{equation*}
f_M(x) :  = \sum_{|j| \leq M}\frac{D^{j}f(0)}{j!} x^j.
\end{equation*}
In particular, for $x \in \mathbb C$,
\begin{equation*}
\left (e^x \right)_M= \sum_{j=0}^M \frac{x^j}{j!}. 
\end{equation*}
We will also use this notation for $f=R$ with $M=2N+5$, and $f= \Omega$ with $M=2N+1$.

\begin{prop}[Local reproducing property]\label{locrepprop}
Let $f \in H^0(B)$, and $c_j$ be the quantities as found in Proposition \ref{coeffexist}. Then for $u \in B$,
\begin{align*}
\begin{split}
f\left(\frac{u}{\sqrt{k}} \right) =& \left\langle \chi_{k}\left(\frac{v}{\sqrt{k}}\right)f\left(\frac{v}{\sqrt{k}}\right),e^{\bar{u}\cdot v}\left(\sum_{j}^N \frac{c_j\left( v, \bar{u} \right)}{\sqrt{k^j}} \right)\right\rangle_{\mathcal{L}^2 \left(B(\sqrt{k}),k\varphi_k(v) \right)} + O\left({k}^{n-\frac{N+1}{2}}\right)\|f\|_{\mathcal{L}^2\left(B,k\varphi \right)}.
\end{split}
\end{align*}
\end{prop}

\noindent To prepare to verify Proposition \ref{locrepprop}, we require a series of estimates on the Taylor series remainders of the exponential term and the volume form.  We also show that the integral outside of the the support of $\chi_k$ is rapidly decaying.

\begin{lemma}[Remainder of the exponential term]\label{remainofexp}
Let $M = [\frac{N+1}{2\epsilon}]+1$. Then for $N \geq 0$ and any $f \in H^0(B)$,
\begin{align*}
\int_{B(\sqrt{k})}&\chi_k \left(\tfrac{v}{\sqrt{k}} \right)f \left(\tfrac{v}{\sqrt{k}} \right)e^{u\bar{v}-|v|^2}\left(\sum_{j =0}^N \tfrac{c_j(u,\oo v)}{\sqrt{k^j}}\right) \left( e^{-kR\left(\frac{u}{\sqrt{k}} \right)} - \left(e^{-kR_{2N+5}\left(\frac{u}{\sqrt{k}} \right)} \right)_{M} \right) \Omega \left( \tfrac{v}{\sqrt{k}} \right) dV \\
&= \|f\|_{\mathcal{L}^2(B,k\varphi)}O \left(k^{n-\tfrac{N+1}{2}}\right).
\end{align*}
\end{lemma}

\begin{lemma}[Remainder of determinant]\label{remofdet} Let $M = [\frac{N+1}{2\epsilon}]+1$. Then
\begin{align}
\begin{split}\label{eq:remofdet1}
&\left| \int_{B(\sqrt{k})}\chi_k \left(\tfrac{v}{\sqrt{k}} \right)f \left(\tfrac{v}{\sqrt{k}} \right)e^{u\bar{v}-|v|^2}\left(\sum_{j =0}^N \tfrac{c_j(u,\oo v)}{\sqrt{k^j}}\right) \left(e^{-kR_{2N+5}\left(\tfrac{v}{\sqrt{k}} \right)} \right)_{M} \left( \Omega - \Omega_{2N+1} \right)\left( \tfrac{v}{\sqrt{k}} \right) dV \right| \\
& \hsp \hsp = \|f\|_{\mathcal{L}^2(B,k\varphi)}O\left(k^{n-\frac{N+1}{2}} \right).
\end{split}
\end{align}
\end{lemma}

\begin{lemma}[Estimate outside the ball]\label{estoutdisk}
Let $F$ be a holomorphic polynomial. Then the following estimate holds:
\begin{align*}
& \int_{\mathbb{C}^n}\left(1-\chi_k \left(\frac{v}{\sqrt{k}} \right)\right)F \left(\frac{v}{\sqrt{k}} \right)e^{u\bar{v}-|v|^2} \left(\sum_{t=0}^N\sum_{m+j =t} \frac{c_j(u,\oo v)a_m(v,\bar{v})}{\sqrt{k^t}}\right) dV\\
& \hsp \leq C_Nk^n\|F\|_{\mathcal{L}^2(B,k\varphi)}e^{-\frac{1}{32}k^{\frac{1}{2}-2\eps}}.
\end{align*}
\end{lemma}

Given the results above, we prove Proposition \ref{locrepprop}.
 
\begin{proof}[Proof of Proposition \ref{locrepprop}] By Proposition \ref{coeffexist},
\begin{equation*}
F \left(\frac{u}{\sqrt{k}} \right) = \int_{\bCn}F \left(\frac{v}{\sqrt{k}}\right)e^{u\cdot\bar{v}-|v|^2}\left(\sum_{t=0}^N\sum_{m+j=t}\frac{c_j(u,\bar{v})a_m(v,\bar{v})}{\sqrt{k^t}}\right)dV,
\end{equation*}
for all $N\geq 0$ and all holomorphic polynomials $F$.  
We then split the above to two pieces.
\begin{align*}
F \left(\frac{u}{\sqrt{k}} \right) &= \int_{\bCn}\left(1-\chi_k \left(\frac{v}{\sqrt{k}} \right) \right)F \left(\frac{v}{\sqrt{k}}\right)e^{u\cdot\bar{v}-|v|^2}\left(\sum_{t=0}^N\sum_{m+j=t}\frac{c_j(u,\bar{v})a_m(v,\bar{v})}{\sqrt{k^t}}\right) dV \\
&+ \int_{\bCn}\chi_k \left(\frac{v}{\sqrt{k}} \right)F \left(\frac{v}{\sqrt{k}}\right)e^{u\cdot\bar{v}-|v|^2}\left(\sum_{t=0}^N\sum_{m+j=t}\frac{c_j(u,\bar{v})a_m(v,\bar{v})}{\sqrt{k^t}}\right) dV.
\end{align*}
The first integral is bounded above by $C_Nk^n\|f\|_{\mathcal{L}^2(B,k\varphi)}e^{-\frac{1}{32}k^{\frac{1}{2}-\eps}}$ from Lemma \ref{estoutdisk}.  For the second integral, we note that since $M=[\frac{N+1}{2\eps}]+1 > N/4$, and $|u|<1$, we have
\begin{align*}
\left| \left(\sum_{t=0}^N\sum_{m+j=t}\frac{c_j(u,\bar{v})a_m(v,\bar{v})}{\sqrt{k^t}}\right) - \left(\sum^N_{j=0}\frac{c_j(u,\bar{v})}{\sqrt{k^j}}\right)\left(e^{-kR_{2N+5}\left(\frac{v}{\sqrt{k}}\right)} \right)_{M}\Omega_{2N+1} \right| \leq C_N k^{-\frac{N+1}{2}}|v|^{4(M+N+1)} .
\end{align*}
Then by applying this to the second integral, and using the Cauchy-Schwarz inequality we get
\begin{align*}
& C_Nk^{-\frac{N+1}{2}}\left|\int_{\mathbb{C}^n}\chi_k\left(\tfrac{v}{\sqrt{k}}\right) F\left(\tfrac{v}{\sqrt{k}}\right)|v|^{4(M+N+1)}e^{u\oo v - |v|^2}dV\right| \\
&\leq C_Nk^{-\frac{N+1}{2}}\left(\int_{\mathbb{C}^n}\chi_k\left(\tfrac{v}{\sqrt{k}}\right)\left|F\left(\tfrac{v}{\sqrt{k}}\right)\right|^2e^{-|v|^2}dV\right)^{\tfrac{1}{2}}\left(\int_{\mathbb{C}^n}\chi_k\left(\tfrac{v}{\sqrt{k}}\right)\left| |v|^{4(M+N+1)}e^{u\bar{v}-\tfrac{|v|^2}{2}}\right|^2dV\right)^{\tfrac{1}{2}} \\
&\leq C_N\|F\|_{\mathcal{L}^2(B,k\varphi)}k^{n-\frac{N+1}{2}}.
\end{align*}
Hence we obtain the estimate,
\begin{align*}
F \left(\frac{u}{\sqrt{k}} \right) &= \int_{\bCn} \chi_k \left(\frac{v}{\sqrt{k}} \right ) F \left(\frac{v}{\sqrt{k}}\right)e^{u\cdot\bar{v}-|v|^2}\left(\sum^N_{j=0}\frac{c_j(u,\bar{v})}{\sqrt{k}}\right)\left(e^{-kR_{2N+5}\left(\frac{v}{\sqrt{k}}\right)} \right)_{M}\Omega_{2N+1}dV \\ &+ O(k^{n-\frac{N+1}{2}})\|F\|_{L^2(B,k\varphi)}.
\end{align*} 
Now by applying Lemma \ref{remofdet} and Lemma \ref{remainofexp}, we have
\begin{align*}
&  \int_{\mathbb{C}^n}\chi_k \left(\tfrac{v}{\sqrt{k}} \right)F \left(\tfrac{v}{\sqrt{k}} \right)e^{u\bar{v}-|v|^2}\left(\sum_{j =0}^N \tfrac{c_j(u,\oo v)}{\sqrt{k^j}}\right) \left(e^{-kR_{2N+5}\left(\tfrac{v}{\sqrt{k}} \right)} \right)_{M} \Omega_{2N+1}\left( \tfrac{v}{\sqrt{k}} \right) dV \\
&=  \left\langle \chi_{k}\left(\tfrac{v}{\sqrt{k}}\right) F \left(\tfrac{v}{\sqrt{k}}\right),e^{\bar{u}\cdot v}\left(\sum_{j}^N \tfrac{c_j\left(\bar{u}, v \right)}{\sqrt{k^j}} \right)\right\rangle_{\mathcal{L}^2(B(\sqrt{k}),k\varphi(\frac{v}{\sqrt{k}}))}  +  \|F\|_{\mathcal{L}^2(B,k\varphi)}O\left(k^{n-\frac{N+1}{2}} \right).
\end{align*}
We can extend to arbitrary $f \in H^0(B)$ by putting $F=f_L$, letting $L\to \infty$, and using the uniform convergence of $f_L$.
The result follows.
\end{proof}

We end this section with the proofs of Lemma \ref{remainofexp}, Lemma \ref{remofdet}, and Lemma \ref{estoutdisk}.

\begin{proof}[Proof of Lemma \ref{remainofexp}]
First note that since $|v| \leq k^{\frac{1}{4}-\eps}$ we have
\begin{equation*}
kR \left(\frac{v}{\sqrt{k}} \right) =  O(k^{-\eps}).
\end{equation*}
We regroup the quantity
\begin{equation}\label{expestimate}
e^{-kR} - e^{-kR_{2N+5}} = e^{-kR}\left(1-e^{k(R-R_{2N+5})} \right).
\end{equation}
By Taylor expansion
\begin{align*}
k \left|(R-R_{2N+5})\left(\frac{v}{\sqrt{k}} \right) \right|  &\leq k\underset{|\xi| \leq \frac{|v|}{\sqrt{k}}}{\underset{|\alpha|=2N+6}{\sup}}\left|\frac{D^{\alpha}R(\xi)}{(\alpha)!}\right|\left|\frac{v}{\sqrt{k}}\right|^{\alpha} \\
& \leq C_Nk\left(\frac{|v|}{\sqrt{k}}\right)^{2N+6}\\
& \leq C_N k^{- \frac{N+1}{2}}.
\end{align*}
Applying the above to \eqref{expestimate}, we have
\begin{equation}\label{expest1}
\left|e^{-kR\left(\frac{v}{\sqrt{k}} \right)} - e^{-kR_{2N+5}\left(\frac{v}{\sqrt{k}} \right)} \right| \leq C_N k^{- \frac{N+1}{2}}.
\end{equation}

Next we consider the difference
\begin{equation}\label{expest2}
\left|e^{-kR_{2N+5}\left(\frac{v}{\sqrt{k}} \right)} - \left(e^{-kR_{2N+5}\left(\frac{v}{\sqrt{k}} \right)} \right)_{M} \right|,
\end{equation}
where $M$ is a fixed constant such that $M \geq \frac{N+1}{2\eps}$.

By $\left(e^{-kR_{2N+5}\left(\frac{v}{\sqrt{k}} \right)} \right)_{M}$ we mean to truncate as
\begin{equation*}
\sum_{j=0}^{M} \frac{1}{j!}\left(-kR_{2N+5}\left(\frac{v}{\sqrt{k}} \right)\right)^j.
\end{equation*}
Hence we have an estimate
\begin{align*}
\left|e^{-kR_{2N+5}\left(\frac{v}{\sqrt{k}} \right)} - \left(e^{-kR_{2N+5}\left(\frac{v}{\sqrt{k}} \right)} \right)_{M} \right| &\leq \sup_{|x| \leq |-kR_{2N+5}\left(\frac{v}{\sqrt{k}} \right)|} \frac{|x|^{M+1}}{(M+1)!} \\
&\leq C k^{-\eps M+1} \\
&\leq Ck^{-\frac{N+1}{2}}.
\end{align*}
Combining \eqref{expest1} and \eqref{expest2}, we have
\begin{equation*}
\left|e^{-kR \left(\frac{v}{\sqrt{k}} \right)} - \left(e^{-kR_{2N+5} \left(\frac{v}{\sqrt{k}} \right)} \right)_{M} \right| \leq C_N k^{-\frac{N+1}{2}}.
\end{equation*}
Applying our estimate directly to the integral,
\begin{align*}
& \left| \int_{\mathbb{C}^n}\chi_k \left(\tfrac{v}{\sqrt{k}} \right)f \left(\tfrac{v}{\sqrt{k}} \right)e^{u\bar{v}-|v|^2}\left(\sum_{j =0}^N \tfrac{c_j(u,\oo v)}{\sqrt{k^j}}\right) \left( e^{-kR\left(\frac{v}{\sqrt{k}} \right)} - \left(e^{-kR_{2N+5}\left(\frac{v}{\sqrt{k}} \right)}\right)_{M}\right) \Omega \left( \tfrac{v}{\sqrt{k}} \right) dV \right| \\
&\leq C_Nk^{-{\frac{N+1}{2}}}\int_{\mathbb{C}^n} \left| \chi_k \left(\tfrac{v}{\sqrt{k}} \right)f \left(\tfrac{v}{\sqrt{k}} \right)e^{-\frac{|v|^2}{2}}\left(\sum_{j =0}^N \tfrac{c_j(u,\oo v)}{\sqrt{k^j}}\right) \Omega \left( \tfrac{v}{\sqrt{k}} \right) e^{u\bar{v}-\frac{|v|^2}{2}}\right| dV \\
&\leq  C_N k^{-{\frac{N+1}{2}}} \left( \int_{\mathbb{C}^n}  \chi_k \left( \tfrac{v}{\sqrt{k}} \right) \left| f \left(\tfrac{v}{\sqrt{k}} \right)\right|^2 e^{-|v|^2} dV \right)^{\frac{1}{2}} \left(\int_{\mathbb{C}^n} \chi_k \left(\tfrac{v}{\sqrt{k}} \right) \left|  \left( \sum_{j =0}^N \tfrac{c_j(u,\oo v)}{\sqrt{k^j}}\right) \Omega \left( \tfrac{v}{\sqrt{k}} \right)  e^{u\bar{v}-\frac{|v|^2}{2}}\right|^2 dV\right)^{\frac{1}{2}} \\
&\leq  C_N k^{-{\frac{N+1}{2}}} \left( \int_{\mathbb{C}^n}  \chi_k \left( \tfrac{v}{\sqrt{k}} \right) \left| f \left(\tfrac{v}{\sqrt{k}} \right)\right|^2 e^{-k \vp \left( \frac{v}{\sqrt{k}} \right)} dV \right)^{\frac{1}{2}} \left(\int_{\mathbb{C}^n} \chi_k \left(\tfrac{v}{\sqrt{k}} \right) \left|  \left( \sum_{j =0}^N \tfrac{c_j(u,\oo v)}{\sqrt{k^j}}\right) \Omega \left( \tfrac{v}{\sqrt{k}} \right)  e^{u\bar{v}-\frac{|v|^2}{2}}\right|^2 dV\right)^{\frac{1}{2}} \\
&\leq C_N \|f\|_{\mathcal{L}^2(B,k\varphi)}k^{n-\frac{N+1}{2}}.
\end{align*}
The result follows.
\end{proof}

\begin{proof}[Proof of Lemma \ref{remofdet}]
We first observe the following estimate
\begin{equation*}
\left| \left( \Omega - \Omega_{2N+1} \right) \left( \frac{v}{ \sqrt{k}}  \right) \right| \leq \underset{|\xi| \leq \left| \frac{v}{\sqrt{k}} \right|}{\sup_{|\ga| = 2N+2}} \left| \frac{D^{\ga} \Omega(\xi)}{\ga !} \right| \left| \frac{v}{\sqrt{k}} \right|^{2N+2} \leq C_N k^{- \frac{N+1}{2}}.
\end{equation*}
Using the above estimate with a similar manipulation as Lemma \ref{remainofexp} we conclude \eqref{eq:remofdet1}.
\end{proof}

\begin{proof}[Proof of Lemma \ref{estoutdisk}] Up this point the estimate \eqref{eqnorms} has been crucial in all of our estimates. However when $\chi_k$ is replaced by $1- \chi_k$, the integrand is not supported in $B(k^{1/4 -\eps})$, and hence \eqref{eqnorms} is not correct anymore. However, an application of integration by parts resolves this issue as we perform below. 

First note that since $|u| \leq 1$ and $|v| \geq \frac{1}{2}k^{\frac{1}{4}-\eps}$, we have
$|u - v| \geq \frac{1}{4}k^{\frac{1}{4}-\eps}$.  Next we use the identity
\begin{equation*}
\dbar \left( \sum_i e^{\bar{v}(u-v)} \frac{1}{u^i-v^i} d\widehat{\oo V^i} \right) = -n e^{\bar{v}(u-v)}dV,
\end{equation*}
where $ d\widehat{\oo V^i} := \left(\frac{\sqrt{-1}}{2\pi}\right)^ndv^1\wedge d\oo v^1 \wedge\ldots \wedge  dv^i \wedge \widehat{d\oo v^i}\wedge \ldots \wedge dv^n \wedge d\oo v^n$.
Integrating by parts, we have
\begin{align*}
&\int_{\mathbb{C}^n}\left(1-\chi_k \left(\frac{v}{\sqrt{k}} \right)\right)F \left(\frac{v}{\sqrt{k}} \right)e^{\bar{v}(u-v)} \left(\sum_{t=0}^N\sum_{m+j =t} \frac{c_j(u,\oo v)a_m(v,\bar{v})}{\sqrt{k^t}}\right) dV \\
&= -\frac{1}{n}\int_{\mathbb{C}^n}\left(1-\chi_k \left(\frac{v}{\sqrt{k}} \right)\right)F \left(\frac{v}{\sqrt{k}} \right)  \left(\sum_{t=0}^N\sum_{m+j =t} \frac{c_j(u,\oo v)a_m(v,\bar{v})}{\sqrt{k^t}}\right)\dbar \left( \sum_i e^{\bar{v}(u-v)} \frac{1}{u^i-v^i} d\widehat{\oo V^i}\right)\\
&= -\frac{1}{n}\int_{\mathbb{C}^n}F \left(\frac{v}{\sqrt{k}} \right) \sum_{i}\dbar_{i}\left( \left(1-\chi_k \left(\frac{v}{\sqrt{k}} \right)\right) \left(\sum_{t=0}^N\sum_{m+j =t} \frac{c_j(u,\oo v)a_m(v,\bar{v})}{\sqrt{k^t}}\right)  \right) e^{\bar{v}(u-v)} \frac{1}{u^i-v^i} dV.
\end{align*}
Iterating the above integration by parts $2N$ times we obtain
\begin{equation*}
= \frac{(-1)^{2N+1}}{n^{2N+1}}\int_{\mathbb{C}^n} F \left(\frac{v}{\sqrt{k}} \right)\underset{|I|={2N+1}}{\sum_{ I = (i_1,\ldots ,i_{2N+1})}} \dd_{\bar{I}} \left(\left(1-\chi_k \left(\frac{v}{\sqrt{k}} \right)\right)
\sum_{t=0}^N \sum_{m+j =t} \frac{c_j(u,\oo v)a_m(v,\bar{v})}{\sqrt{k^t}} \right) \frac{e^{\bar{v}(u-v)}}{(u-v)^I} dV.
\end{equation*}
Since the degrees of $a_m$ and $c_j$ are $2m$ and $2j$ respectively, always one differentiation is applied to $1-\chi_k$.  Therefore, the integrand is supported on the annulus $\frac{1}{2}k^{\frac{1}{4}-\eps} \leq |v| \leq k^{\frac{1}{4}-\eps}$.
The above integral is then bounded above by
\begin{align*}
&\left(\int_{\frac{1}{2}k^{\frac{1}{4}-\eps} \leq |v| \leq k^{\frac{1}{4}-\eps}} \left| F \left(\frac{v}{\sqrt{k}} \right) \right|^2 e^{-|v|^2}dV\right)^{\frac{1}{2}} \\
&\times \left( \int_{\frac{1}{2}k^{\frac{1}{4}-\eps} \leq |v| \leq k^{\frac{1}{4}-\eps}} \left| \underset{|I|={2N+1}}{\sum_{ I = (i_1,\ldots ,i_{2N+1})}} \dd_{\bar{I}} \left(\left(1-\chi_k \left(\frac{v}{\sqrt{k}} \right)\right)
\sum_{t=0}^N \sum_{m+j =t} \frac{c_j(u,\oo v)a_m(v,\bar{v})}{\sqrt{k^t}} \right) \frac{e^{\frac{|u|^2-|u-v|^2}{2}}}{(u-v)^I} \right|^2 dV \right)^{\frac{1}{2}} \\
&\leq C_N k^n\|F\|_{L^2(B,k\varphi)}e^{-\frac{1}{32}k^{\frac{1}{2}-2\eps}}.
\end{align*}
The result follows
\end{proof}

\section{Local to global}\label{localtoglobal}

The norm of $K$ as a section of the bundle $L^{\otimes k} \ten \bar{L}^{\ten k}$ is the \textit{Bergman function $\mathcal{B}$}. Hence in local coordinates
\begin{equation*}
\mathcal{B}(x) = |K(x,x)|_{h^k} = |\tilde{K}(x,x)|e^{-k\varphi(x)},
\end{equation*}
where $\tilde{K}(x,x)$ is the coefficient function of the Bergman kernel with respect to the frame $e_L^{\ten k} \ten \bar{e_L}^{\ten k}$. We also have an extremal characterization of the Bergman function given by
\begin{equation}\label{extremal}
\mathcal{B}(x) = \sup\limits_{\|s\|_{\mathcal{L}^2} \leq 1} |s(x)|_{h^k}^2,
\end{equation}
where $s \in H^0(M,L^{\ten k})$.

\begin{lemma}[Uniform upper bound on Bergman function]\label{unifupbound}
 There exists $C$ dependent on $M$, and independent of $k$ and $x$ such that
\begin{equation*}
\mathcal{B}(x)\leq Ck^n.
\end{equation*}

\begin{proof}
We use the extremal characterization of the Bergman function \eqref{extremal}. On the compact manifold $M$, we fix a finite coordinate cover $\{U\}$ and also fix a coordinate $(z_1,\cdots,z_n)$ in $U$. For each $U$, we have a local K\"ahler potential $\varphi(z)$. We can assume that $U=B(0,2)$, and $\sup_{z\in B(0,2)}|D^2\varphi(z)|\leq C$, i.e. the second derivatives are uniformly bounded. Since $\varphi$ is plurisubharmonic, we can assume the volume form $dV_g=(\tfrac{\sqrt{-1}}{2\pi}\partial\bar{\partial}\varphi)^n(z)$ is equivalent to $dV_E(z)$ in $B(0,2)$.
\begin{align*}
1
&\geq \int_{B(z_0,\frac{1}{\sqrt{k}})} |\tilde{s}(z)|^2e^{-k\varphi(z)}dV_g\\
&\geq \frac{1}{C_1}\int_{B(z_0,\frac{1}{\sqrt{k}})} |\tilde{s}(z)|^2e^{-k\varphi(z)}dV_E\\
&\geq \frac{1}{C_1}\exp(-\sup_{B(0,2)}|D^2\varphi|)\int_{B(z_0,\frac{1}{\sqrt{k}})} |\tilde{s}(z)|^2e^{-k\varphi(z_0)-k\varphi_z(z_0)(z-z_0)-k\varphi_{\bar{z}}(z_0)\overline{(z-z_0)}}dV_E\\
&=\frac{1}{C_2} e^{-k\varphi(z_0)}\int_{B(z_0,\frac{1}{\sqrt{k}})} |\tilde{s}(z)e^{-k\varphi_z(z_0)(z-z_0)}|^2dV_E.
\end{align*}
Since $\tilde{s}(z)e^{-k\varphi_z(z_0)(z-z_0)}$ is holomorphic, by the mean-value inequality we have
\begin{align*}
&\frac{1}{C_2}e^{-k\varphi(z_0)}\int_{B(z_0,\frac{1}{\sqrt{k}})} |\tilde{s}(z)e^{-k\varphi_z(z_0)(z-z_0)}|^2dV_E\\
&\geq \frac{1}{C_3k^n} e^{-k\varphi(z_0)}|\tilde{s}(z_0)|^2.
\end{align*}
So we have $$e^{-k\varphi(z_0)}|\tilde{s}(z_0)|^2\leq C_3 k^n, $$ where $C_3$ is uniform for any $z_0\in B(0,1)$ and any $s\in H^0(M,L^k)$.
Taking the supremum over all such $s$ and a standard finite cover argument yields the desired result.
\end{proof}
\end{lemma}

Let $K(x,y) = K_{y}(x)$ be the global Bergman kernel of $H^0(M,L^{\ten k})$.  We view $K_y(x)$ as a section of $L^k \otimes \oo L_y^k$.  We shall use $ \tilde{K}(x, y)$ for the local representation of $K(x,y)$ with respect to the frame $e_L(x)^{\ten k} \ten \bar{e_L}(y)^{\ten k}$.

\begin{thm}[Local to global] The following equality relates the truncated local Bergman kernel $$K_N^{loc}\left(\frac{u}{\sqrt{k}}, \frac{v}{\sqrt{k}}\right)  =k^n e^{u\cdot\bar{v}}\sum_{j=0}^N \frac{c_j(u, \oo v)}{\sqrt{k}^j}$$ to the global Berman kernel $\tilde K$.
\begin{equation*}
\tilde K \left(\frac{u}{\sqrt{k}},\frac{v}{\sqrt{k}} \right) = K_{N}^{loc}\left(\frac{u}{\sqrt{k}},\frac{v}{\sqrt{k}} \right) + O \left(k^{2n-\frac{N+1}{2}} \right).
\end{equation*}

\begin{proof}
Fix $u, v \in B$. We apply the local reproducing property to the global Bergman kernel $f(w)=\tilde K_{\frac{u}{\sqrt{k}}}(w)=\tilde K (w, \frac{u}{\sqrt{k}})$
\begin{equation*}
\tilde{K} \left( \frac{v}{\sqrt{k}},\frac{u}{\sqrt{k}} \right) = \left\langle \chi_k(w)\tilde{K} \left(w,\frac{u}{\sqrt{k}} \right) , K_N^{loc} \left(w,\frac{v}{\sqrt{k}} \right)\right\rangle_{\mathcal{L}^2(B,k\varphi(w))} +  O \left(k^{n-\frac{N+1}{2}} \right)\|\tilde{K}_{\frac{u}{\sqrt{k}}}\|_{\mathcal{L}^2(B,k\varphi)}.
\end{equation*}
By the reproducing property, we obtain from Lemma \ref{unifupbound},
\begin{equation*}
\|\tilde{K}_{\frac{u}{\sqrt{k}}}\|^2_{\mathcal{L}^2(B,k\varphi)}\leq \|K_{\frac{u}{\sqrt{k}}}\|^2_{\mathcal{L}^2}=\tilde{K}\left(\frac{u}{\sqrt{k}},\frac{u}{\sqrt{k}} \right)=B \left(\frac{u}{\sqrt{k}} \right)e^{k\varphi_k(u)} \leq Ck^n,
\end{equation*}
where $K_{\frac{u}{\sqrt{k}}}(w)$ means section with respect to $w$ and local coefficient function with respect to $u$. Thus we have
\begin{equation*}
\tilde{K} \left(\frac{v}{\sqrt{k}},\frac{u}{\sqrt{k}} \right) = \left\langle \chi_k(w)\tilde{K} \left(w,\frac{u}{\sqrt{k}} \right) , K_N^{loc}\left(w,\frac{v}{\sqrt{k}} \right) \right\rangle_{\mathcal{L}^2(B,k\varphi(w))} + O\left(k^{2n-\frac{N+1}{2}}\right).
\end{equation*}
We next estimate the difference of the local Bergman kernel with the projection of the local kernel.
\begin{align*}
\begin{split}
g_{k,v} \left(\frac{w}{\sqrt{k}} \right)
:=& \chi_k \left(\frac{w}{\sqrt{k}} \right) K^{loc}_N \left(\frac{w}{\sqrt{k}}, \frac{v}{\sqrt{k}}\right) - \overline{\left\langle \chi_k(x)\tilde{K} \left(x,\frac{w}{\sqrt{k}} \right) , K_N^{loc}\left(x,\frac{v}{\sqrt{k}} \right)\right\rangle}_{\mathcal{L}^2(B,k\varphi(x))}\\
= & \chi_k \left( \frac{w}{\sqrt{k}} \right) K_N^{loc} \left(\frac{w}{\sqrt{k}}, \frac{v}{\sqrt{k}} \right) - \left\langle \chi_k(x)K_N^{loc}\left(x,\frac{v}{\sqrt{k}} \right),{\tilde K} \left(x, \frac{w}{\sqrt{k}} \right)\right\rangle_{\mathcal{L}^2(B,k\varphi(x))}.
\end{split}
\end{align*}
We can regard $g_{k,v}$ as a global section of $L^{\ten k}$ because of the cut-off function $\chi_k$. Since
\begin{equation*}
\left\langle \chi_k K_{N, \frac{v}{\sqrt{k}}}^{loc},{K}_{\frac{u}{\sqrt{k}}} \right\rangle_{\mathcal{L}^2} = \mathcal{P}_{H^0}\left(\chi_k K_{N,\frac{v}{\sqrt{k}}}^{loc}\right),
\end{equation*}
where $\mathcal{P}_{H^0}$ is the Bergman projection and $g_{k,v}$ is the $\mathcal{L}^2$-minimal solution to
\begin{equation*}
\dbar g_{k,v} = \dbar \left(\chi_k K_{N,\frac{v}{\sqrt{k}}}^{loc} \right).
\end{equation*}
Now we estimate $\dbar (\chi_k K_{N,\tfrac{v}{\sqrt{k}}}^{loc})$.
\begin{align*}
\left. \dbar \left( \chi_k K_{N, \frac{v}{\sqrt{k}}}^{loc} \right) \right|_{\frac{w}{\sqrt{k}}} &= \left. \left( \dbar \left( \chi_k \right) K_{N,\frac{v}{\sqrt{k}}}^{loc}   + \chi_k \dbar\left( K_{N,\tfrac{v}{\sqrt{k}}}^{loc}\right) \right) \right|_{\frac{w}{\sqrt{k}}} \\
&=\left. \dbar\left( \chi_k \right) K_{N,\tfrac{v}{\sqrt{k}}}^{loc}  \right|_{\frac{w}{\sqrt{k}}}.
\end{align*}
We note that $\dbar\left( K_{N,\tfrac{v}{\sqrt{k}}}^{loc}\right) =0$ because $K_N^{loc}$ is holomorphic.  The term $\dbar\left( \chi_k \right)$ on the right hand side ensures $|w-v|\geq \frac{1}{4}k^{\frac{1}{4}-\eps}$, and $|w| \leq k^{\frac{1}{4}-\eps}$. Furthermore, since $K_{N,\frac{v}{\sqrt{k}}}^{loc}(\frac{w}{\sqrt{k}}) =  O \left(e^{w\bar{v}} |vw|^{2N}\right)$, and because
\begin{equation*}
|e^{w\bar{v}}|^2 e^{-|w|^2} =e^{2\Real{w\bar{v}}-|w|^2}=e^{-|w-v|^2+|v|^2}\leq Ce^{-\frac{1}{16}k^{\frac{1}{2}-2\eps}},
\end{equation*}
we obtain
\begin{equation*}
\left\|\dbar(\chi_k)K_{N,\frac{v}{\sqrt{k}}}^{loc} \right\|_{\mathcal{L}^2(M,L^{\ten k})} \leq Ce^{-\frac{1}{32} k^{\frac{1}{2}-2\eps}}.
\end{equation*}
So by the Horm\"ander's $\mathcal{L}^2$-estimate, the following inequality holds uniformly for $v \in B$,
\begin{equation}\label{gkvdelta}
\|g_{k,v}\|_{\mathcal{L}^2(M,L^{\ten k})}\leq C e^{-\frac{1}{32}  k^{\frac{1}{2}-2\eps}}.
\end{equation}
By the same argument as  in the Lemma \ref{unifupbound} above, for all $u \in B$ we obtain the uniform estimate
\begin{equation*}
\left|g_{k,v}\left(\tfrac{u}{\sqrt{k}} \right) \right|\leq C e^{-\frac{1}{64} k^{\frac{1}{2}-2\eps}}.
\end{equation*}
This concludes the estimate
\begin{equation*}
\left|K_N^{loc}\left(\frac{u}{\sqrt{k}}, \frac{v}{\sqrt{k}} \right)-K \left(\frac{u}{\sqrt{k}}, \frac{v}{\sqrt{k}} \right) \right|\leq C  k^{2n-\frac{N+1}{2}},
\end{equation*}
uniformly for all $u, v \in B$.
\end{proof}
\end{thm}

\section{Proof of Proposition \ref{coeffexist}\label{proofofprop}} As we said before at the heart of the proof we have Lemma \ref{MainLemma}, namely the following identity:
\begin{equation*}
\int_{\bCn}{\bar{v}^{p}v^q e^{u \cdot \bar{v} - |v|^2} dV} =
\begin{cases}
0 & \text{ if } p > q, \\
\frac{q!}{(q-p)!} u^{q-p} & \text{ if } p \leq q.
\end{cases}
\end{equation*}
To prove Proposition \ref{coeffexist}, it is sufficient to consider arbitrary degree $l$ monomials in $u$. By Lemma \ref{MainLemma}, the proof of our proposition is reduced to
\begin{align}\label{eq:reppropmon}
\begin{split}
u^l &= \int_{\mathbb{C}^n}v^le^{u\cdot\bar{v}-|v|^2} \sum_{t=0}^N\sum_{j+m =t}\left(\sum_{p,q}\sum_{r,s}\frac{c_j^{p,q}a_{m}^{r,s}}{\sqrt{k^t}}u^p\bar{v}^qv^r\oo v^s\right)dV\\
&=\sum_{t=0}^N\sum_{m+j =t} \sum_{q+s \leq l+r}\sum_{p} \frac{c^{p,q}_ja^{r,s}_m}{\sqrt{k^t}}u^{p+l+r-q-s}\frac{(l+r)!}{(l+r-q-s)!}.
\end{split}
\end{align}
Note that in the above summation, $l$ is fixed.

We can immediately determine the $c_0$ coefficients, as seen in the following lemma.

\begin{lemma}\label{csub0}
For multiindices $p,q \in \mathbb{Z}_+^{n}$ the following property holds
\begin{equation}\label{eq:c0id}
c_0^{p,q} = \begin{cases}
1 & \mbox{if }p_i=q_i =0 \text{ for all }i ,\\
0 & \mbox{otherwise}.
\end{cases}
\end{equation}

\begin{proof}
The proof proceeds by induction on the length of the multiindex $q$. First, we consider $| q |= 0$. Then taking \eqref{eq:reppropmon} for $l = (0,\ldots, 0)$ and comparing the coefficient of $k$ coefficient yields
\begin{equation*}
1= \sum_{\underset{q+s \leq r}{p}} c^{p,q}_0a^{r,s}_0 u^{p+r-q-s}\frac{r!}{(r-q-s)!}.
\end{equation*}
By the vanishing of $a_0^{r,s}$ from Proposition \ref{remacoeff}, we have
\begin{equation*}
1= \sum_{p} c^{p,0}_0 u^p.
\end{equation*}
Immediately we compare the coefficients of $u$ and conclude result \eqref{eq:c0id} for this case.

Now we assume the induction hypothesis holds for $|q| \leq \lambda-1$ and consider the case $|q|=\lambda$, and take $|l| = \lambda$. Then applying the induction hypothesis to \eqref{eq:reppropmon} and parsing apart the right hand summation yields
\begin{align*}
u^{l}&= \sum_p \sum_{q \leq l} u^p c_0^{p,q} \frac{l!}{(l-q)!} u^{l-q} \\
&= \sum_p \underset{q \leq l}{\sum_{|q| =\lambda}} u^p c_0^{p,q}\frac{l!}{(l-q)!} u^{l-q}  + \sum_p \underset{|q| \leq \lambda-1}{\sum_{q \leq l}} u^p c_0^{p,q} \frac{l!}{(l-q)!} u^{l-q}.
\end{align*}
Note in particular that the requirements in the first right side term that $q \leq l$ and $|q| =\lambda$ immediately imply that $q=l$. Additionally, by the induction hypothesis, the second right side term reduces to simply $u^l$. Subtracting this from both sides yields
\begin{equation*}
0 = \sum_p u^p c_0^{p,l}.
\end{equation*}
The coefficients vanish accordingly and we conclude \eqref{eq:c0id}. The desired result follows.
\end{proof}
\end{lemma}

\begin{proof}[Proof of Proposition \ref{coeffexist}]
By \eqref{eq:reppropmon}, from comparing the coefficients of $k$, we get
\begin{equation}\label{eq:ulcases}
\sum_{m+j=t}\sum_{s+q\leq l+r}\sum_p c_j^{p,q} a_m ^{r,s}u^{p+r+l-q-s}\frac{(l+r)!}{(l+r-q-s)!}=
\begin{cases}
u^l & t=0,\\
0 & t \geq 1.
\end{cases}
\end{equation}

In order to determine the coefficients $c_j^{p,q}$, we induct on $j$.  The base case $j =0$ is demonstrated by Lemma \ref{csub0}. Now assume the induction hypothesis is satisfied for $j\leq\tau -1$, which implies that the coefficients $c_j^{p,q}$ have been determined for all multiindices $p$ and $q$ and for all such values of $j$.

Take \eqref{eq:ulcases} for $t=\tau$,
\begin{equation*}
\sum_{j=0}^{\tau}\sum_{s+q\leq l+r}\sum_p c_j^{p,q} a_{\tau-j} ^{r,s}u^{p+r+l-q-s}\frac{(l+r)!}{(l+r-q-s)!}=0.
\end{equation*}
By moving all the terms with $j\leq \tau-1$ to the other side, we have
\begin{equation*}
\sum_{s+q\leq l+r}\sum_p c_{\tau}^{p,q} a_{0} ^{r,s}u^{p+r+l-q-s}\frac{(l+r)!}{(l+r-q-s)!}=-\sum_{j=0}^{\tau-1}\sum_{s+q\leq l+r}\sum_p c_j^{p,q} a_{\tau-j} ^{r,s}u^{p+r+l-q-s}\frac{(l+r)!}{(l+r-q-s)!}.
\end{equation*}
Since we know $a_0^{r,s}$ all vanish except that $a_0^{0,0}=1$ (c.f. Lemma \ref{remacoeff}),
\begin{equation}\label{eq:exstcoeff1}
\sum_{q\leq l}\sum_p c_{\tau}^{p,q} u^{p+l-q}\frac{l!}{(l-q)!}=-\sum_{j=0}^{\tau -1}\sum_{s+q\leq l+r}\sum_p c_j^{p,q} a_{\tau-j} ^{r,s}u^{p+r+l-q-s}\frac{(l+r)!}{(l+r-q-s)!}.
\end{equation}
We consider various lengths of $l$ to determine the values of the coefficients. When $|l|=0$, \eqref{eq:exstcoeff1} reduces to
\begin{equation}\label{eq:exstcoeff2}
\sum_p c_{\tau}^{p,0} u^{p}=-\sum_{j=0}^{\tau-1}\sum_{s+q\leq r}\sum_p c_j^{p,q} a_{\tau-j} ^{r,s}u^{p+r-q-s}\frac{r!}{(r-q-s)!}.
\end{equation}
Since, the values of $a_m^{r,s}$ are determined (c.f. Lemma \ref{remacoeff}) and $c_j^{p,q}$ are known due to the inductive hypothesis for $j\leq \tau -1$, we obtain all $c_{\tau}^{p,0}$ by comparing the coefficients of $u$ in \eqref{eq:exstcoeff2}.
We begin a subinduction argument on the values of $|q|$ such that $c_{\tau}^{p,q}$ is known for any $|q| \leq \lambda-1$. The case $\lambda =0$ is determined via our analysis of \eqref{eq:exstcoeff2} discussed above. Consider multiindices $l$ such that $|l| = \lambda$ within the \eqref{eq:exstcoeff3}. As in the Lemma \ref{csub0}, we decompose the left hand summation of \eqref{eq:exstcoeff1} and rearrange the equality to obtain
\begin{equation}\label{eq:exstcoeff3}
\sum_p c_{\tau}^{p,l}u^p = -\sum_{j=0}^{\tau -1}\sum_{s+q\leq l+r}\sum_p c_j^{p,q} a_{\tau-j} ^{r,s}u^{p+r+l-q-s}\frac{(l+r)!}{(l+r-q-s)!} - \sum_p \sum_{\underset{|q| \leq \lambda -1}{q \leq l}}c_{\tau}^{p,q}u^{p+l-q}\frac{l!}{(l-q)!}.
\end{equation}
Due to the induction hypothesis on $\tau$, the first quantity of the right hand side is completely determined. Furthermore, by comparing the coefficients on $u^p$ we can solve for $c^{p,l}_\tau$. This concludes the induction on $|q|$ which implies that $c_{\tau}^{p,q}$ are completely determined, and thus the induction on $\tau$ is also completed. The result follows.
\end{proof}

Next we prove the finiteness of the expansion of the volume form.

\begin{defn}[Weight of coefficient]
Let $A(z) = \sum_{j} \frac{a_j(z)}{\sqrt{k^j}}$ be a series such that each coefficient $a_j(z)$ is a polynomial. Define the \emph{weight} $w(a_j(z))$ of each coefficient as
\begin{equation*}
w(a_j(z)) = \deg(a_j(z)) - j,
\end{equation*}
and the weight of a series $w(A(z)) = \underset{j}{\sup} \ w(a_j(z))$.
\end{defn}

\begin{lemma}[Additivity of weight]
For series $A(z)$ and $B(z)$, we have
\begin{align*}
&w(A(z) + B(z)) = \sup\{w(A(z)),w(B(z))\},\\
&w(A(z)B(z)) = w(A(z)) + w(B(z)).
\end{align*}
\end{lemma}

\begin{proof}
The first equality is clear.  For the second, by direct computation,
\begin{align*}
w(A(z)B(z)) &= w(\left(\sum_j a_j(z)\right)\left(\sum_k b_k(z)\right))\\
&= w(\sum_{t}\sum_{j+k = t}a_j(z)b_k(z)) \\
&= \underset{t}{\sup} \ w(\sum_{j+k=t} a_j(z)b_k(z)) \\
&= \underset{t}{\sup} \ \deg(\sum_{j+k= t}a_j(z)b_k(z)) - t \\
&= \underset{j}{\sup} \ \deg(a_j(z)) - j + \underset{k}{\sup} \ \deg(b_k(z)) - k = w(A(z))+w(B(z)).
\end{align*}
The results follow.
\end{proof}

We now consider the weight of $e^{kR\left(\frac{v}{\sqrt{k}} \right)}$, as a series in powers of $\sqrt{k}$.  Since $R\left(\frac{v}{\sqrt{k}}\right)$ is obtained from Taylor expansion, the degree of the polynomial coefficient matches with the degree of $\sqrt{k}$, so
$w(R(\frac{v}{\sqrt{k}})) = 0$.  Using the additivity of the weight, we have
\begin{equation*}
w(kR(\frac{v}{\sqrt{k}})) = 2.
\end{equation*}
Then we have
\begin{equation*}
e^{kR\left(\frac{v}{\sqrt{k}} \right)} = \sum_{m}\frac{1}{m!}\left(kR\left(\frac{v}{\sqrt{k}} \right)\right)^m.
\end{equation*}
Computing the weight of each term, we obtain
\begin{equation*}
w\left(\left(kR\left(\frac{v}{\sqrt{k}} \right)\right)^m \right) = 2m.
\end{equation*}
Hence we obtain the following lemma.

\begin{lemma} \label{finitecoeffa}
Given $m \in \mathbb{Z}_+$ and multiindices $r,s \in \mathbb{Z}^n_+$ such that $|r+s| > 2m$, then $a^{r,s}_m = 0$.
\end{lemma}

Next we will show that the coefficients $c_j^{p,q}$ and $a_m^{r,s}$ satisfy a parity condition.  We begin by establishing the following property of bivariate power series.
\begin{defn}[Parity property of power series]\label{parprop}
We say the coefficients of the bivariate power series
\begin{equation*}
B(x,y) := \sum_{m,p,q}\frac{B^{p,q}_m}{\sqrt{k^m}}x^py^q,
\end{equation*}
has the \textit{parity property} if given $p,q \in \mathbb{Z}_+^n$ with $|p|+|q| \nmodt m$, then $B^{p,q}_m  = 0$.
\end{defn}

By demonstrating that this property is preserved under standard algebraic manipulations, we will conclude the vanishing of particular coefficients $a_j^{p,q}$ and $c_j^{p,q}$ of \eqref{eq:acoeff} and \eqref{ccoeff}, respectively. This is key in demonstrating the finiteness and bounds on the degrees of each $c_j$.

\begin{lemma}\label{parityprop}
If $A$ and $B$ are bivariate power series with the parity property, then so are $A+B$ and $AB$.

\begin{proof}
The additive closure is immediate. For the multiplication, we have
\begin{equation*}
A(x,y)B(x,y) = \sum_{m=0}^{\infty}\frac{1}{\sqrt{k^m}}\sum_{n=0}^{m}\sum_{p,q,r,s}A^{p,q}_nB^{r,s}_{m-n}x^{p+r}\oo y^{q+s}.
\end{equation*}
The term is nonzero when $|p|+|q| \modt n$ and $|r|+|s| \modt m-n$, that is, when
\begin{equation*}
|p|+|r|+|q|+|s| \modt m.
\end{equation*}
The result follows.
\end{proof}
\end{lemma}

\begin{lemma} The bivariate expansion of $e^{-kR} \Omega$ 
has the parity property.

\begin{proof}
Consider the expansion
\begin{align*}
e^{-kR\left(\frac{v}{\sqrt{k}} \right)} &= \sum_n \left( (-1)^n(\sqrt{k})^2 \sum\frac{1}{\sqrt{k^m}} R_{m}^{p,q} v^p \oo v^q \right)^n
\\
 &= (-1)^n(\sqrt{k})^{2n} \sum_n \left( \sum \frac{1}{\sqrt{k^m}} R_{m}^{p,q} v^p \oo v^q \right)^n.
\end{align*}
The factors of $k^{-1/2}$ come from evaluating the expansion of the exponential at $(\tfrac{v}{\sqrt{k}})$, the $R^{p,q}_m$ terms have the parity property.  Since the parity property is closed under addition and multiplication by Lemma \ref{parityprop}, and multiplication by $k^{2n}$ preserves the parity property, the entire power series admits the property.  Furthermore since $\Omega(\tfrac{v}{\sqrt{k}})$ also has the parity property, the product also has the parity property. The result follows.
\end{proof}
\end{lemma}

\begin{lemma}\label{csub0is0} For all $m \in \mathbb{N}$, given $p,q \in \mathbb{Z}_+^n$ such that $|p|+|q| \nmodt m$, we have $c_m^{p,q}= 0$.

\begin{proof}
The proof proceeds by induction on $m$ satisfying the Lemma statement. First, the case when $m=0$ is an immediate consequence of \eqref{eq:c0id}. Next assume that the parity property holds for $m \leq t-1$ and we do a sub induction on $|q|$. From \eqref{eq:exstcoeff2} we obtain
\begin{equation*}
- \sum_{\alpha} c_{t}^{\ga, 0} u^{\ga} = \sum_{j=0}^{t-1}\sum_p \sum_{r \geq q + s}{c_j^{p,q} a_{t-j}^{r,s} u^{p+r-q-s} \frac{r!}{(r-q-s)!}} .
\end{equation*}
On the right hand side, the coefficients are nonzero only when $|p|+|q| \modt j$ and $|r| + |s| \modt t - j$. Combining these two equalities yields
\begin{equation*}
|p|+|q| +|r| +|s| \modt t,
\end{equation*}
which implies that $c_{t}^{\alpha, 0}= 0$ for $\alpha \nmodt t$.

Now we assume the induction hypothesis holds for $|q| \leq \lambda-1$. For $|q| = \lambda$, we have from \eqref{eq:exstcoeff3},
\begin{equation*}
\sum_p c_{t}^{p,l}u^p = -\sum_{j=0}^{t -1}\sum_{s+q\leq l+r}\sum_p c_j^{p,q} a_{t-j} ^{r,s}u^{p+r+l-q-s}\frac{(l+r)!}{(l+r-q-s)!} - \sum_p \sum_{\underset{|q| \leq \lambda -1}{q \leq l}}c_{t}^{p,q}u^{p+l-q}\frac{l!}{(l-q)!}.
\end{equation*}
On the right hand side, in the first summation, when $|p| + |q| +|r| + |s| \not\modt t$, the terms are zero by the induction on $j$, hence the exponent must be
\begin{equation*}
|p|+|r|+|q|+|s| + |l| \modt t +|l|.
\end{equation*}
In the second summation, when $|p| + |q|  \not\modt t$, the terms are zero by the induction on $|q|$, hence the exponent must be 
\begin{equation*}
|p|+|q|+ |l| \modt t +|l|.
\end{equation*}
Then comparing the exponent on both sides, we get on the left hand side,
\begin{equation*}
|p|\modt t+|l|.
\end{equation*}
Hence
\begin{equation*}
|p|+|l|\modt t.
\end{equation*}
The subinduction on $|q|$ has been proven, therefore we may determine all $c_t^{p,q}$ for the given $t$. Consequently the induction on $t$ is complete, and the result follows.
\end{proof}
\end{lemma}

We next establish two combinatorial identities in preparation for the proof of Theorem \ref{pq2mcsubmpq0}.

\begin{lemma}\label{bimonid1s}
Given $l \in \mathbb{Z}_+$, for all $s \in [1,l] \cap \mathbb{Z}_+$ the following equality holds
\begin{equation}\label{eq:bimonid1s}
\sum_{w=0}^s(-1)^s \binom{l}{w} \binom{l-w}{l-s}= 0.
\end{equation}

\begin{proof}
We assign the terms of \eqref{eq:bimonid1s} as polynomial coefficients and rearrange terms appropriately,
\begin{align*}
\sum_{w=0}^{l} \sum_{s=0}^w(-1)^s \binom{l}{w} \binom{l-w}{l-s} x^{l-s} &= \sum_{w=0}^{l} \left( \sum_{s=w}^{l} \binom{l-w}{l-s} x^{l-s}\right) \binom{l}{w} (-1)^s \\
&= \sum_{w=0}^l \left( (x+1)^{l-w} \right) \binom{l}{w} (-1)^w \\
 &= \sum_{w=0}^l{\left((x+1)^{l-w} \binom{l}{w}(-1)^w \right)} \\
 & = x^l.
 \end{align*}
 The result follows.
\end{proof}
\end{lemma}

\begin{cor}\label{binomid1} Given a multiindex $l \in \mathbb{Z}_+^n$, for all nonzero multiindices $s \leq l$ the following equality holds
\begin{equation}\label{eq:binomid1}
\sum_{w \leq s} (-1)^{|w|} \binom{l}{w} \binom{l-w}{l-s} = 0.
\end{equation}

\begin{proof} Taking the left hand side of \eqref{binomid1} and decomposing it as a product of binomial coefficients we obtain
\begin{align*}
\sum_{w \leq s} (-1)^{|w|} \binom{l}{w} \binom{l-w}{l-s}  & =\prod_{i=1}^{n} \left(\sum_{w_i \leq s_i} (-1)^{w_i} \binom{l_i}{w_i} \binom{l_i-w_i}{l_i-s_i} \right).
\end{align*}
We observe that since $s$ is nonzero then there is at least one $i \in [1,n] \in \mathbb{Z}_+$ such that $s_i$ is nonzero. Applying Lemma \ref{bimonid1s} to this index within the product yields the desired result.
\end{proof}
\end{cor}

\begin{lemma}\label{binomid2s}
For all $\eta \in [1,l] \cap \mathbb{Z}_+$, and $r \in [0,\eta-1] \cap \mathbb{Z}_+$,
\begin{equation}\label{eq:binomid2s}
\sum_{w=0}^{\eta}(-1)^w \binom{l}{w} \binom{r+l-w}{\eta-w} = 0.
\end{equation}

\begin{proof} To verify the lemma we apply the following combinatorial identity. For $0 \leq w \leq 2$, we have
\begin{equation*}
\binom{r+l-w}{\eta - w} =\left( \binom{l-w}{\eta - w} \binom{r}{0} + \binom{l-w}{\eta - w-1} \binom{r}{1} + \cdots + \binom{\eta + 1 - w}{0} \binom{r}{\eta - w} \right).
\end{equation*}
We expand left hand side of \eqref{eq:binomid2s} by the above identity and obtain
\begin{align*}
\sum_{w=0}^{\eta}(-1)^w \binom{l}{w} \sum_{v = 0}^{\eta - w} \binom{l-w}{\eta - w - v} \binom{r}{v} &= \sum_{v = 0}^{\eta} \left( \sum_{w=0}^{\eta - v}(-1)^w \binom{l}{w} \binom{l-w}{\eta - w - v} \right)_L \binom{r}{v}.
\end{align*}
It suffices to prove that the labeled quantity $L = 0$, that is
\begin{equation*}
\sum_{w=0}^{\eta - v}(-1)^w \binom{l}{w} \binom{l - w}{\eta - j - v} = 0.
\end{equation*}
This is equivalent to demonstrating that for any $\eta \in \mathbb{Z}_+$ such that $1 \leq \eta \leq l +1$,
\begin{equation}\label{eq:binomid2s}
r_{\eta} = \sum_{w=0}^{\eta} (-1)^w \binom{l}{w} \binom{l -w}{\eta- w} = 0.
\end{equation}
We again embed \eqref{eq:binomid2s} as coefficients of a polynomial in $x$ and with careful manipulation obtain,
\begin{align*}
\sum_{\eta =0}^{l}{r_{\eta}x^{\eta}}&=\sum_{\eta = 0}^{l} \sum_{w=0}^{\eta} (-1)^w \binom{l}{w} \binom{l-w}{\eta -w} x^{\eta}\\
&= \sum_{w=0}^{l}\left( \sum_{\eta =w}^{l} \binom{l-w}{\eta - w} x^{\eta -w} \right) \binom{l}{w}(-x)^w\\
&= \sum_{w=0}^{l}(1+x)^{l - w} \binom{l}{w} (-x)^w\\
&= 1.
\end{align*}
The result follows.
\end{proof}
\end{lemma}

\begin{cor}\label{binomid2}
For all multiindices $\eta ,r \in \mathbb{Z}_+^n$ with $r < \eta$,
\begin{equation}\label{eq:binomid2}
{\sum_{w \leq \eta}  (-1)^{|w|+1} \binom{l}{w}  \binom{r+l-w}{\eta - w}} = 0.
\end{equation}

\begin{proof} Taking the left hand side of \eqref{binomid2} and decomposing it as a product of binomial coefficients we obtain
\begin{align*}
{\sum_{w \leq \eta}  (-1)^{|w|+1} \binom{l}{w}  \binom{r+l-w}{\eta - w}}  & =- \prod_{i=1}^{n} \left({\sum_{w_i \leq \eta_i}  (-1)^{w_i} \binom{l_i}{w_i}  \binom{r_i+l_i-w_i}{\eta_i - w_i}} \right).
\end{align*}
We observe that since $s$ is nonzero then there is at least one $i \in [1,n] \in \mathbb{Z}_+$ such that $r_i<\eta_i$. Applying Lemma \ref{binomid2s} to this index within the product yields the desired result.
\end{proof}
\end{cor}

\begin{thm}\label{pq2mcsubmpq0} Given $m \in \mathbb{Z}_+$ and multiindices $p,q \in \mathbb{Z}_+^n$ such that $|p+q| > 2m$, then $c_m^{p,q} = 0$.

\begin{proof} Recall the identity  \eqref{eq:exstcoeff1} established in Lemma \ref{coeffexist}, we have
\begin{equation}\label{eq:rightnleft}
\sum_{q\leq l}\sum_p c_{\tau}^{p,q} u^{p+l-q}\frac{l!}{(l-q)!}=-\sum_{j=0}^{\tau -1}\sum_{s+q\leq l+r}\sum_p c_j^{p,q} a_{\tau-j} ^{r,s}u^{p+r+l-q-s}\frac{(l+r)!}{(l+r-q-s)!}.
\end{equation}
With the theorem statement as the induction hypothesis, for each fixed $\tau \in \mathbb{N}$, we induct on appropriate values of $m$. For $m=0$ this is true by \eqref{eq:c0id}.
Next assume the hypothesis holds when $m \leq \tau-1$. Then we consider the case $m=\tau$. Define the quantities (given by the left and right hand sides of \eqref{eq:rightnleft} respectively),
\begin{equation}\label{keyrel1}
\mathcal{P}_l:= \sum_{q \leq l} \sum_{p} c_{\tau}^{p,q} u^{p+l - q} \frac{l!}{(l-q)!},
\end{equation}
and
\begin{equation}\label{keyrel2}
\mathcal{Q}_l := - \sum_p\sum_{r + l \geq q+s}\sum_{{j=0}}^{\tau-1} c_j^{p,q} a_{\tau- j}^{r,s} u^{p + r+l -q -s} \frac{(r+l)!}{(r+l-q-s)!}.
\end{equation}
Then we have that by \eqref{eq:rightnleft} that $\mathcal{P}_{l} = \mathcal{Q}_l$ for all multiindices $l \in \mathbb{Z}_+^n$. We prove the identity $c_m^{p,q}=0$ by embedding the two families of coefficients $\{ \mathcal{P}_{w}\}_{w \leq l}$ and $\{ \mathcal{Q}_{w}\}_{w \leq l}$ into a polynomial.  Set, for $B \in \{ \mathcal{P}, \mathcal{Q} \}$,
\begin{equation}\label{eq:Psidef}
\Psi_{l}(B):= \sum_{w \leq l}{(-1)^{|w|}  \binom{l}{w} u^w B_{l-w}}.
\end{equation}
First we compute $\Psi_{l}(\mathcal{P})$ by inserting \eqref{keyrel2} into \eqref{eq:Psidef}, carefully rearranging terms with respect to powers of $u$:
\begin{align*}
\Psi_{l}(\mathcal{P}) &= \sum_{w \leq l }{(-1)^{|w|} \binom{l}{w} u^w \mathcal{P}_{l-w}} \\
&=  \sum_{w \leq l}{(-1)^{|w|} \sum_{p }\sum_{q \leq l-w} \binom{l}{w} c_{\tau}^{p,q} u^{l+ p - q} \frac{(l-w)!}{(l-w-q)!}} \\
&=\sum_{p}\sum_{w \leq l} (-1)^{|w|} \left( \sum_{w \leq s \leq l} c_{\tau}^{p,(l-s)} u^{p+s} \binom{l}{w} \frac{(l-w)!}{(s-w)!} \right) \\
&=\sum_{p}\sum_{s \leq l} c_{\tau}^{p,(l-s)} u^{p+s}\left( \sum_{w \leq s} (-1)^{|w|} \binom{l}{w} \frac{(l-w)!}{(s-w)!} \right).
\end{align*}
Note the substitution used to obtain the second to last line is through the identification $s = l -q$, and then interchanging the order of summation yields the final line. We decompose the above summation into two pieces to conclude that
\begin{equation*}
\Psi_{l}(\mathcal{P}) =\sum_{p} c_{\tau}^{p,l}   u^{p} l! + \sum_{p}\sum_{0 \neq s \leq l} c_{\tau}^{p,(l-s)} u^{p+s}\left( \sum_{w \leq s} (-1)^{|w|} \binom{l}{w} \frac{(l-w)!}{(s-w)!} \right).
\end{equation*}
Applying Corollary \ref{binomid1},  we have that the quantity above reduces to simply
\begin{equation}\label{PsilPshort}
\Psi_{l}(\mathcal{P}) =\sum_{p} c_{\tau}^{p,l}   u^{p} l!.
\end{equation}
We next compute $\Psi_{l}(\mathcal{Q})$. We collect up terms with respect to the powers of $u$,
\begin{align}
\begin{split}\label{eq:PsilQ1}
\Psi_l(\mathcal{Q}) &=  \sum_{w \leq l }{(-1)^{|w|} \binom{l}{w} u^w \mathcal{Q}_{l-w}} \\
&=  - \sum_{w \leq l }{(-1)^{|w|} \binom{l}{w} u^w \sum_p\sum_{r + l \geq q+s+w} \sum_{{j=0}}^{\tau-1} c_j^{p,q} a_{\tau- j}^{r,s} u^{p + r+l-w -q -s} \frac{(r+l-w)!}{(r+l-w-q-s)!}}\\
&=  \sum_{{j=0}}^{\tau-1}\sum_{w \leq l }{\sum_p\sum_{r + l \geq q+s+w}  (-1)^{|w|+1} \binom{l}{w} c_j^{p,q} a_{\tau- j}^{r,s} u^{p + r+l -q -s} \frac{(r+l-w)!}{(r+l-w-q-s)!}}\\
&=\sum_{{j=0}}^{\tau-1}\sum_{q+s=r}^{r+l}{\sum_p\sum_{w\leq r+l-q-s}  (-1)^{|w|+1} \binom{l}{w} c_j^{p,q} a_{\tau- j}^{r,s} u^{p + r+l -q -s} \frac{(r+l-w)!}{(r+l-w-q-s)!}}.
\end{split}
\end{align}
For simplicity, set $\eta := r+ l -q-s$ and allow it to range $0 \leq \eta \leq l$. Updating the index of \eqref{eq:PsilQ1} we obtain
\begin{align}
\begin{split}\label{thmid1}
\Psi_l(\mathcal{Q}) &=\sum_{{j=0}}^{\tau-1}\sum_p\sum_{q+s=r}^{r+l}{\sum_{w \leq \eta}  (-1)^{|w|+1} \binom{l}{w} c_j^{p,q} a_{\tau- j}^{r,s} u^{p + \eta} \frac{(r+l-w)!}{(\eta - w)!}}\\
&=\sum_{{j=0}}^{\tau-1}\sum_p\sum_{q+s=r}^{r+l}c_j^{p,q} a_{\tau- j}^{r,s} u^{p + \eta}{\sum_{w \leq \eta}  (-1)^{|w|+1} \binom{l}{w}  \frac{(r+l-w)!}{(\eta - w)!}}.
\end{split}
\end{align}
As a result of Corollary \ref{binomid2} we then only need to consider $r \geq \eta$, otherwise the term vanishes. As a result of the induction hypothesis combined with the fact that $a_j^{r,s} = 0$ when $r+s > 2j$ (cf. Remark \ref{remacoeff}) we apply these facts to $\Psi_l(\mathcal{Q})$,
\begin{align*}
2\tau & \geq |p + q + r+s|
= |p + 2r + l - \eta|.
\end{align*}
Manipulating the above expression yields
\begin{align*}
|p+\eta| & \leq 2\tau - 2 |r| - |l| + 2 |\eta | \leq 2 \tau - 2|(r-\eta)| - |l |\leq 2\tau -|l|.
\end{align*}
Combining this fact with \eqref{thmid1} we conclude
\begin{align*}
\deg \Psi_{l}(\mathcal{Q}) \leq 2\tau - |l|.
\end{align*}
Recall that since $\mathcal{P}_l = \mathcal{Q}_l$ for all $l$ by \eqref{eq:rightnleft}, so that $\Psi_l(\mathcal{P}) = \Psi_l(\mathcal{Q})$. Noting that as a result of \eqref{PsilPshort}, we have that, the index $p$ must satisfy
\[ |p| \leq \deg \Psi_l(\mathcal{Q}) \leq 2 \tau - |l|,\]
therefore if $2 \tau< |p| + |l|$, then $c_{\tau}^{p, l} = 0$, demonstrating the desired induction step. The result follows.
\end{proof}
\end{thm}
\section{Computation of the coefficients}\label{compofcoeff}
We next explicitly compute the coefficient $c_1$ and $c_2$ of $K^{loc}$ under Bochner coordinates (the coefficients $c_0$ was computed in Lemma \ref{csub0}.

To compute $c_1$ and $c_2$ we require preliminary terms of the \k potential as well as the coefficients $a_m^{r,s}$.

\begin{prop}[Expansion of \k potential]\label{expkahlerpot}
We have the following series expansion of the potential $\varphi$ under Bochner coordinates is given by
\begin{align*}
\varphi(z) &= |z|^2 - \frac{\Rm_{i\oo j k\oo l}(0)}{4}z^iz^k\oo z^j \oo z^l + O(|z|^5) \\
&= |z|^2 + R(z,\bar{z}).
\end{align*}
\end{prop}

\begin{lemma}[Properties of $e^{-kR}\Omega$ expansion]\label{remacoeff}
The expansion up to $\frac{1}{k}$ for $\sum_{m}\sum_{p,q}\frac{a_m^{p,q}v^p\oo v^q}{\sqrt{k^m}}$ is
\begin{equation*}\label{eq:remacoeff}
e^{-k R\left(\frac{v}{\sqrt{k}} \right)} \Omega \left(\tfrac{v}{\sqrt{k}} \right) = 1 - \tfrac{1}{k} \left(\Ric_{k\oo l}v^k \oo v^l - \tfrac{1}{4} \Rm_{i\oo j k\oo l}(0) v^iv^k\oo v^j \oo v^l \right),
\end{equation*}
where the numbers $a_j^{p,q}$ for $j=0, 1, 2$ are given by
\begin{equation*}
a_0^{p,q} =
\begin{cases}
1 &\text{ if } |p| = |q| = j = 0 \\
 0 & \text{ otherwise},
\end{cases}
\end{equation*}
\begin{equation*}
a_1^{p,q} = 0 \ \ \ \text{ for all } p,q,
\end{equation*}
and lastly
\begin{equation*}
a_2^{p,q} =
\begin{cases}
- \sum_{k,l} \Ric_{k \bar{l}} &\text{ if } |p| = |q| =1 \\
 \tfrac{1}{4} \sum_{k,l} \Ric_{i \bar{j} k \bar{l}}(0) &\text{ if } |p| = |q| =2\\
 0 & \text{ otherwise}.
\end{cases}
\end{equation*}

\begin{proof}
We expand each quantity of the product on the left hand side of \eqref{eq:remacoeff}. First, for the exponential term we have
\begin{equation*}
e^{-kR\left(\frac{v}{\sqrt{k}} \right)} = 1 + \frac{ \Rm_{i\oo j k\oo l}(0)}{4k}v^iv^k\oo v^j \oo v^l + o(k^{-\frac{3}{2}}).
\end{equation*}
And the determinant quantity becomes
\begin{align*}
\Omega \left(\frac{v}{\sqrt{k}} \right) &= \det\left( \delta_{i\oo j} + \frac{1}{k}\frac{\dd^4 \varphi_{i\oo j}}{\dd z^k \dd \oo z^l}(0)v^k\oo v^l  + O\left(k^{-\frac{3}{2}} \right) \right) \\
&= 1 - \frac{1}{k}\Ric_{k\oo l}v^k \oo v^l + O\left(k^{-\frac{3}{2}} \right).
\end{align*}
The result follows.
\end{proof}
\end{lemma}

The computation of $c_1$ is now immediate.

\begin{cor}\label{csub1is0}
For all $p,q \in \mathbb{Z}_+$ we have $c_1^{p,q} = 0$.

\begin{proof}
By comparing coefficients in \eqref{eq:reppropmon}, we see that there is no contribution from $a_1$ for the $\frac{1}{\sqrt{k}}$ term, hence
\begin{equation*}
\frac{c_1^{p,q}}{\sqrt{k}} u^{p+l-q} = 0,
\end{equation*}
for any $l$.  The result follows.
\end{proof}
\end{cor}
\subsection{Computing the coefficient $c_2$}
By applying equation \eqref{eq:reppropmon} to $f(z) = 1$ we obtain $c_2^{0,0}$:
\begin{equation*}
1 = \int_{\mathbb{C}^n} e^{u\cdot\bar{v} - |v|^2}\left( 1 + \frac{c_2^{0,0}}{k} + \frac{c^{(i),(j)}_2}{k}u^i\oo v^j + \frac{c^{(i,k),(j,l)}_2}{k}u^iu^k\oo v^j \oo v^l \right) \left(1 -  \tfrac{1}{k} \left(\Ric_{k\oo l}v^k \oo v^l - \tfrac{1}{4} \Rm_{i\oo j k\oo l}(0) v^iv^k\oo v^j \oo v^l \right) \right)dV.
\end{equation*}
Collecting the $\frac{1}{k}$ terms, we obtain
\begin{equation*}
c_2^{0,0} = \int_{\mathbb{C}^n}\Ric_{i\oo j}v^i\oo v^j e^{u\oo v -|v|^2}dV - \frac{1}{4}\int_{\mathbb{C}^n}\Rm_{i\oo jk\oo l}v^iv^k\oo v^j \oo v^l e^{u\oo v -|v|^2}dV.
\end{equation*}
The first integral on the right hand side is nonzero when $i = j$. The left side is nonzero when $i = j, k = l$ and $i = l, j = k$.  We therefore obtain
\begin{equation*}
c_2^{0,0} = \frac{\rho}{2}.
\end{equation*}
Next to obtain the $c^{(i)(j)}_2$ coefficient, we apply equation (\ref{eq:reppropmon}) with $f= v^\alpha$ to obtain
\begin{equation*}
\sum_{i}c^{(i)(\alpha)}_2 u^i = \int_{\mathbb{C}^n}\left(\Ric_{k\oo l}v^kv^{\alpha}\oo v^l - \frac{1}{4}\Rm_{i\oo jk\oo l}v^iv^kv^\alpha \oo v^j\oo v^l\right)e^{u\cdot v - |v|^2}dV.
\end{equation*}
The first term on the right hand side is nonzero when $\alpha = l$, hence the only relevant term after integration is
\begin{equation*}
\int_{\mathbb{C}^n}\Ric_{i\oo \alpha}v^i|v^{\alpha}|^2 e^{u\cdot v - |v|^2}dV = \sum_i\Ric_{i\oo \alpha}u^i.
\end{equation*}
The second term splits into four cases:
\begin{enumerate}
\item $\alpha = j, i = l$.
\item $\alpha = j, k = l$.
\item $\alpha = l, i = j$.
\item $\alpha = l, k = j$.
\end{enumerate}
In each case, after integration, we obtain the term
\begin{equation*}
\int_{\mathbb{C}^n}\Rm_{i\oo \alpha,k,\oo i}v^k|v^i|^2|v^{\alpha}|^2 e^{u\cdot v - |v|^2}dV = \sum_{k}\Ric_{k\oo \alpha}u^k,
\end{equation*}
and similar computations for the other cases, hence
\begin{equation*}
c_2^{(i)(\alpha)} = 0.
\end{equation*}
Next to obtain the $c^{(ik),(jl)}_2$ coefficient, we apply equation \eqref{eq:reppropmon} with $f= v^\alpha v^\beta$ to obtain
\begin{equation*}
2\sum_{i,k}c^{(ik),(jl)}_2 u^i u^k = \int_{\mathbb{C}^n}\left(\Ric_{k\oo l}v^kv^{\alpha}v^\beta \oo v^l - \frac{1}{4}\Rm_{i\oo jk\oo l}v^iv^kv^\alpha v^\beta \oo v^j\oo v^l\right)e^{u\cdot v - |v|^2}dV.
\end{equation*}
For the first term on the right hand side, it is not possible to sum over two variables, hence is an irrelevant term.  The second term has two cases:
\begin{enumerate}
\item $j = \alpha, l = \beta$.
\item $l = \alpha, j = \beta$.
\end{enumerate}
Hence we have
\begin{equation*}
c^{(ik),(\alpha\beta)}_2 = -\frac{1}{4}R_{i\oo\alpha k\oo\beta}.
\end{equation*}
Note that the result matches with \cite{lushiff} except for the emergence of non-analytic terms, however the computations in Lu and Shiffman were done for the lifted Szeg\"o kernel.
\section{Higher order convergence}
As the $C^m$ norms depend on the choice of coordinates, we must give some care when discussing the convergence in higher order.  The local kernel that we have constructed is an expansion at one point $p \in M$.  We now show the regularity of the local kernel depending on the point $p$.  

We have shown that at a point $p \in M$
\begin{equation*}
|K_k(p + z, p + w) - K_{(N)}^{loc}(p,z,w)| \leq \frac{C_{p,N}}{k^{N+1-2n}}, \ \ \ \ d(z,w) < \frac{1}{\sqrt{k}}.
\end{equation*}
In fact, the $C_{p,N}$ depends on the local potential, that is,
\begin{equation*}
C_{p,N} \leq \sup_{\underset{x\in B_p(2\delta)}{|\alpha| \leq \alpha(N)}} |D^\alpha \varphi(x)|
\end{equation*}

We first would like to show that given a point $q \in B_p(\delta)$, the constant $C_{p,N}$ is uniform in that neighborhood, i.e.

\begin{equation*}
|K_k(q + z, q + w) - K_{(N)}^{loc}(q,z,w)| \leq \frac{C_{p,N}}{k^{N+1-2n}}, \ \ \ \ d(z,w) < \frac{1}{\sqrt{k}}.
\end{equation*}

Consider a smooth family of B\"ochner coordinates.  The existence of such a coordinate is given, for example in \cite{liulu}. Then consider a finite cover of $M$ by $B_p(2\delta)$ of fixed radius.  Then for $q \in B_p(2\delta)$, we have
\begin{equation*}
\sup_{B_q(\delta)}|D^\alpha \varphi| \leq C \sup_{B_p(2\delta)}|D^\alpha \varphi|,
\end{equation*}
where $C$ is independent of $q$, and the derivatives $D^\alpha$ on the left correspond to the B\"ochner coordinates centered at $p$ and the right corresponds to the B\"ochner coordinates centered at $q$.

To show the convergence for higher order derivatives with respect to the variable $p$, we first apply the B\"ochner-Martinelli formula to the difference of the local kernel and the global kernel.

We recall that
\begin{lemma}[B\"ochner-Martinelli kernel]\label{bochmart}
For $w,z \in \mathbb{C}^n$, we define the Bochner-Martinelli kernel, $M(w,z)$ 
\begin{equation*}
M(w,z) = \frac{(n-1)!}{(2\pi\sqrt{-1})^n}\frac{1}{|z-w|^{2n}}\sum_{1\leq j \leq n}(\oo w^j-\oo z^j)d\oo{w}^1\wedge dw^1 \wedge \cdots \wedge dw^j \wedge \cdots \wedge d\oo{w}^n \wedge dw^n
\end{equation*}
Suppose that $f \in C^\infty(D)$ where $D$ is a domain in $\mathbb{C}^n$ with piecewise smooth boundary.  Then for $z \in D$,
\begin{equation*}
f(z) = \int_{\dd D} f(w)M(w,z) - \int_D \dbar f(w)\wedge M(w,z).
\end{equation*}
\end{lemma}

Now let $p \in M$ and consider B\"ochner coordinates $(z^1,\cdots,z^n)$ centered at $p$.  The Bergman kernel and the local kernel are both objects that depend on the base point and two arguments,i.e.
\begin{equation*}
K_k(p,z,w) := K_k(p+z, p+w)
\end{equation*} 
By polarizing in the $p$ variable and considering the almost holomorphic extension, we may view the kernel as 
\begin{equation*}
K_k(p,q,z,w) := K_k(p+z,q+w)
\end{equation*}
Let 
\begin{equation*}
f_k(p,q,z,w) = K_k(p,q,z,w) - K_{(N)}^{loc}(p,q,z,w)
\end{equation*}
be the difference between the global and local kernel.  Note that $f_k$ is defined for $q,p+z,q+w \in B_p(\frac{1}{\sqrt{k}})$.  From our previous result, we have
\begin{equation*}
|f_k(p,p,z,w)| \leq \frac{C_{p,N}}{k^{{N+1-2n}}}, \ \ \ \ d(z,w) < \frac{1}{\sqrt{k}}.
\end{equation*}
 
We want to estimate
\begin{equation*}
|\dd_p^\alpha f_k(p,q)|
\end{equation*}
for $q = p$, where we suppress the $z,w$ variable because it it not essential to the argument. Applying Lemma \ref{bochmart} to $\dd_p^\alpha f_k(p,q)$ with $D = B_p(\frac{1}{\sqrt{k}}) \times B_q(\frac{1}{\sqrt{k}})$, we obtain
\begin{equation*}
\dd_p^\alpha f_k(p,q) =  \int_{\dd D} f_k(p',q')\dd_p^\alpha M(p',q',p,q) - \int_D \dbar f_k(p',q')\wedge \dd_p^\alpha M(p',q',p,q).
\end{equation*}
The boundary integral term can be bounded by the $\mathcal{L}^\infty$-norm of $f_k$ multiplied by $\sqrt{k}^{-|\alpha|}$.  By using the fact that $f_k$ is an almost holomorphic extension, $\dbar \dd^\alpha_p f_k$ in the second integral is bounded by $O_{\alpha}(|q'-p'|^\infty)$. When $p=q$, we have $d(p',q') < \frac{1}{\sqrt{k}}$, and therefore the second integral is of order $O(k^{-\infty})$.

Now we show the higher order convergence with respect to the $z,w$ variable. We rescale $z \mapsto \frac{u}{\sqrt{k}}$ and $w \mapsto \frac{v}{\sqrt{k}}$ to match the notation as in the statement of our theorem. Since the local kernel $K^{loc}_N$ and the global Bergman kernel are holomorphic in $u$ and anti-holomorphic in $v$, the derivatives can be bounded by the $L^\infty$-norms using Cauchy estimates. More precisely, let $D_x$ be any first order differential operator of $x$.  By using the Cauchy estimates on $K^{loc}_N(x,y)$ and $\tilde{K}(x,y)$ on the ball of radius $\frac{1}{\sqrt{k}}$, we obtain
\begin{align*}
\left|D_x(K^{loc}_N(x,\tfrac{v}{\sqrt{k}}) - \tilde{K}(x,\tfrac{v}{\sqrt{k}}))\right| &\leq C\sqrt{k}\|K^{loc}_N(\cdot, \tfrac{v}{\sqrt{k}}) - \tilde{K}(\cdot, \tfrac{v}{\sqrt{k}})\|_{\mathcal{L}^\infty(B(k^{-(1/2)}))} \\
& = O(k^{2n+\tfrac{1}{2}-\frac{N+1}{2}}).
\end{align*}
The above holds for $x \in B(\tfrac{1}{2}k^{-1/2})$, hence we have
\begin{equation*}
\left|D_x(K^{loc}_N(\tfrac{u}{\sqrt{k}},\tfrac{v}{\sqrt{k}}) - \tilde{K}(\tfrac{u}{\sqrt{k}},\tfrac{v}{\sqrt{k}}))\right| = O(k^{2n+\tfrac{1}{2}-\frac{N+1}{2}}).
\end{equation*}
By similar argument, we can obtain the same estimates for the holomorphic variables $\bar{y}$. 
% Note that we choose the ball of radius $k^{-1/2}$ in order to use $\mathcal{L}^\infty$ estimate of the Bergman kernel.  

Now let $D^\alpha$ be any $\alpha$-th degree differential operator with respect to $x$ or $\bar{y}$. By iterating the previous argument, we obtain the following
\begin{equation*}
|D^\alpha(K^{loc}_N - K)| \leq O(k^{\tfrac{|\alpha|}{2}+2n-\tfrac{N+1}{2}}).
\end{equation*}
Hence we obtain the smooth convergence of the Bergman kernel asymptotics.

\section{Appendix}\label{s:appendix}
In this appendix we discuss the Bargmann-Fock space $\mathcal{F} $. It is the space of entire functions that satisfy the weighted square integrability condition:
\begin{equation*}
\int_{\mathbb{C}^n}|f(z)|^2e^{-|z|^2}dV < \infty.
\end{equation*}
The space $\mathcal{F}$ is precisely $H^0(\mathbb{C}^n,|z|^2)$, and is thus a closed linear subspace of the space $\mathcal{L}^2(\mathbb{C}^n, |z|^2)$ with inner product given by
\begin{equation*}
\langle f, g \rangle_{\mathcal{F}} := \int_{\mathbb{C}^n}f(z)\overline{g(z)}e^{-|z|^2}dV,
\end{equation*}
and thus is a Hilbert space.  In fact, it is a \textit{reproducing kernel Hilbert space} on $\mathbb{C}^n$, with reproducing kernel
\begin{equation*}
\mathcal{R}_{\bCn}(u,v) := e^{u \cdot \bar{v}}.
\end{equation*}
We first show that this kernel has the reproducing property on $\mathbb{C}$ and then extend this argument to $\mathbb{C}^n$.

\begin{lemma}\label{repkerC} On $\bC$, the Bargmann-Fock kernel is given by
\begin{equation*}
\mathcal{R}_{\bC}(u,v) := e^{u \bar{v}}.
\end{equation*}

\begin{proof} Taking some $f \in H^0(\bCn)$, we consider the inner product against $\mathcal{R}_{\bC}$. We convert the resulting integral to polar coordinates and then apply the Cauchy Integral Formula to obtain
\begin{align*}
\left\langle f(v), \mathcal{R}_{\bC} \right\rangle_{\mathcal{F}} &=  \sqrt{-1}\int_{\bC}{f(v) e^{u \bar{v}- |v|^2} \frac{dv \w d\bar{v}}{2 \pi} }\\
&=  - \frac{1}{\pi} \int_{0}^\infty \int_0^{2 \pi} {f(u + r e^{i \theta}) e^{u (\bar{u} + re^{-i\theta})- |u + re^{i \theta}|^2} \frac{r}{2} d\theta  dr}\\
&= - \frac{1}{\pi} \int_{0}^\infty r e^{-r^2} \int_0^{2 \pi} {f(u + r e^{i \theta}) e^{ - \bar{u} r e^{i \theta} }  d\theta  dr}\\
&=  -  f(u) \int_{0}^\infty{ 2 r e^{-r^2} dr}\\
& = f(u).
\end{align*}
The result follows.
\end{proof}
\end{lemma}

\begin{cor}
On $\mathbb{C}^n$, the Bargmann-Fock kernel is given by
\begin{equation*}
\mathcal{R}_{\bCn}(u,v) := e^{u \cdot \bar{v}}.
\end{equation*}

\begin{proof}
Let $u,v \in \mathbb{Z}_+^n$ with $u = (u_1,\ldots,u_n)$ and $v = (v_1, \ldots, v_n)$. Observe that
\begin{equation*}
e^{u \cdot \bar{v}} =\prod_{i=1}^n{e^{u_i\bar{v_i} - |v_i|^2}}.
\end{equation*}
To demonstrate the reproducing property, we consider $f \in H^0(\mathbb{C}^n)$ and decompose the integrand of the resulting inner product agains $\mathcal{R}_{\bCn}$. Applying Lemma \ref{repkerC} to each dimensional component, we have
\begin{align*}
\langle f(v), \mathcal{R}_{\mathbb{C}^n} \rangle_{\mathcal{F}} &= \int_{\bCn}{f(v) e^{u \cdot v - |v|^2} dV} \\
&= \int_{\bCn}f(v_1,\ldots,v_n)\left( \prod_{i=i}^n{e^{u_i\bar{v_i} - |v_i|^2}} \right) dV \\
&= f(u).
\end{align*}
The result follows.
\end{proof}
\end{cor}

The following lemmas demonstrate the Bargmann-Fock projection of monomials of different variables.

\begin{lemma}
Given some multiindex $m \in \mathbb{Z}_+^n$ the following equality holds.
\begin{equation*}
\int_{\bCn}{\bar{v}^{m} e^{u \cdot \bar{v} - |v|^2} dV} = 0.
\end{equation*}

\begin{proof}
By manipulation and an application of Dominated Convergence Theorem,
\begin{align*}
\int_{\bCn}{\bar{v}^{m} e^{u \cdot \bar{v} - |v|^2} dV} &= \int_{\bCn}{ \del_u^{(m)} \left[  e^{u \cdot \bar{v} - |v|^2}  \right] dV} \\
&= \del_u^{(m)}  \left[ \int_{\bCn}{ e^{u \cdot \bar{v} - |v|^2}  dV} \right] \\
&= 0.
\end{align*}
Note that the integral is constant with respect to $u$, hence the derivative vanishes. The result follows.
\end{proof}
\end{lemma}

\begin{lemma}\label{MainLemma}
The following equality holds, for $p,q \in \mathbb{Z}_+$ with $p \leq q$.
\begin{equation*}
\int_{\bCn}{\bar{v}^{p}v^q e^{u \cdot \bar{v} - |v|^2} dV} =
\begin{cases}
0 & \text{ if } p > q, \\
\frac{q!}{(q-p)!} u^{q-p} & \text{ if } p \leq q.
\end{cases}
\end{equation*}
\end{lemma}

\begin{proof}
Again by manipulation and an application of Dominated Convergence Theorem,
\begin{align*}
\int_{\bCn}{\bar{v}^{p}v^q e^{u.\bar{v} - |v|^2} dV} &= \int_{\bCn}{ \del_u^{(p)} \left[  v^q e^{u \cdot \bar{v} - |v|^2}  \right] dV} \\
&= \del_u^{(p)}  \left[ \int_{\bCn}{ v^q e^{u \cdot \bar{v} - |v|^2}  dV} \right] \\
&= \del_u^{(p)}  \left[ u^q \right],
\end{align*}
therefore
\begin{equation*}
\int_{\bCn}{\bar{v}^{p}v^q e^{u \cdot \bar{v} - |v|^2} dV} =
\begin{cases}
0 & \text{ if } p > q, \\
\frac{q!}{(q-p)!} u^{q-p} & \text{ if } p \leq q.
\end{cases}
\end{equation*}
Result follows.
\end{proof}

\subsection*{Acknowledgements}

The first author is partially supported by the NSF grant DMS-0969745. The second author is supported by the NSF graduate fellowship DGE-1321846. The second author thanks Jeffrey Streets for his constant support and mentoring. The third and fourth authors thank Zhiqin Lu for his helpful discussions and guidance. The third author would also like to thank Prof. Chin-Yu Hsiao for his friendly discussions and taking time to explain some recent developments in the field.

\end{document}